\pdfoutput=1
\documentclass{shinyart}

\usepackage[utf8]{inputenc}

\usepackage{shinybib}

\usepackage{autonum}
\usepackage{enumerate}

\renewcommand{\phi}{\varphi}
\newcommand{\N}{\mathbb{N}}
\newcommand{\R}{\mathbb{R}}
\newcommand{\calF}{\mathcal{F}}
\newcommand{\calG}{\mathcal{G}}
\newcommand{\calI}{\mathcal{I}}
\newcommand{\calJ}{\mathcal{J}}
\newcommand{\calA}{\mathcal{A}}

\newcommand{\calS}{\mathcal{S}}
\newcommand{\scalprod}[1]{\langle #1 \rangle}
\newcommand{\norm}[1]{\| #1 \|}
\newcommand{\set}[2]{\left\{#1:#2\right\}}
\DeclareMathOperator{\sign}{\mathrm{sign}}
\DeclareMathOperator{\dom}{\mathrm{dom}}

\DeclareMathOperator{\Id}{\mathrm{Id}}

\newcommand{\prox}{\mathrm{prox}}

\usepackage{tikz,pgfplots}
\usepgfplotslibrary{colorbrewer}
\pgfplotsset{compat=newest}
\pgfplotsset{plot coordinates/math parser=false}

\title{A convex analysis approach to optimal controls with switching structure for partial differential equations}
\author{Christian Clason\thanks{Faculty of Mathematics, University Duisburg-Essen, 45117 Essen, Germany (\email{christian.clason@uni-due.de})}
    \and Kazufumi Ito\thanks{Department of Mathematics, North Carolina State University, Raleigh, North Carolina, USA (\email{kito@math.ncsu.edu}).} 
\and Karl Kunisch\thanks{Institute of Mathematics and Scientific Computing, University of Graz, Heinrichstrasse 36, 8010 Graz, Austria (\email{karl.kunisch@uni-graz.at}).}}

\date{March 15, 2015}

\hypersetup{
    pdftitle={A convex analysis approach to optimal controls with switching structure for partial differential equations},
    pdfauthor={Christian Clason, Kazufumi Ito, Karl Kunisch},
    pdfkeywords={Optimal control, switching control, partial differential equations, nonsmooth optimization, convexification, semi-smooth Newton method}
}

\begin{document}

\maketitle
\allowdisplaybreaks

\begin{abstract}
    Optimal control problems involving hybrid binary--continuous control
    costs are challenging due to their lack of convexity and weak lower
    semicontinuity. Replacing such costs with their convex relaxation leads
    to a primal-dual optimality system that allows an explicit pointwise
    characterization and whose Moreau--Yosida regularization is amenable to
    a semismooth Newton method in function space. This approach is
    especially suited for computing switching controls for partial
    differential equations. In this case, the optimality gap between the
    original functional and its relaxation can be estimated and shown to be
    zero for controls with switching structure. Numerical examples
    illustrate the effectiveness of this approach. 
\end{abstract}

\section{Introduction}
\label{sec:introduction}

In the context of control of differential equations, switching control refers to problems with two or more controls of which only one should be active at every point in time. This is a challenging problem due to its hybrid discrete--continuous nature.

To partially set the stage, consider the parabolic partial differential equation $Ly = Bu$ on $\Omega_T:=[0,T]\times \Omega$, where $L=\partial_t - A$ for an elliptic operator $A$ defined on $\Omega\subset \R^n$, and $B$ is defined by $(Bu)(t,x)= \chi_{\omega_1}(x) u_1(t)+\chi_{\omega_2}(x)u_2(t)$ for given control domains $\omega_1,\omega_2\subset\overline\Omega$ (which may include controls acting on the boundary).
To promote a switching structure, we propose to use the binary function
\begin{equation}\label{eq:def_penalty}
    |\cdot|_0:\R\to \R,\qquad |t|_0 := \begin{cases} 
        1 & \text{if } t \neq 0,\\
        0 & \text{if } t = 0,
    \end{cases}
\end{equation}
to construct a cost functional which has the value $0$ if and only if at most one control is active pointwise. To guarantee coercivity, we also need to add an (in this case) quadratic term, i.e., we define for $v=(v_1,v_2)\in\R^2$ the pointwise control cost
\begin{equation}
    g(v) = \frac\alpha2(v_1^2 + v_2^2) + \beta |v_1 v_2|_0.
\end{equation}
This term combines in a single functional both switching enhancement and a quadratic cost for the active control(s), where the binary part naturally acts as a penalization of the switching constraint $v_1 v_2 = 0$. In this respect we shall consider the asymptotic behavior $\beta \to \infty$ in \cref{sec:optimality}.

For some $\omega_T\subset \Omega_T$ we then consider the problem 
\begin{equation}\label{eq:problem_motiv}
    \left\{\begin{aligned}
            \min_{u \in L^2(0,T;\R^2)} &\frac12\norm{y-z}_{L^2(\omega_T)}^2 + \int_0^T g(u(t))\,dt,\\
            \text{s.\,t.}\quad & Ly = Bu.
    \end{aligned}\right.
\end{equation}
Using the solution operator $S=L^{-1}B:u \mapsto y$, problem \eqref{eq:problem_motiv} can be expressed in reduced form~as
\begin{equation}\label{eq:orig}
    \min_{u} \calF(u) + \calG(u),
\end{equation}
where $\calF$ is smooth and convex, and $\calG$ is neither smooth nor convex nor, in fact, weakly lower semicontinuous (since this is the case if and only if $g$ is lower semicontinous and convex, which is not the case; see, e.g., \cite[Corollary 2.14]{Braides}). This makes both its analysis and its numerical solution challenging; for example, one cannot rely on standard techniques to guarantee existence of solutions. We therefore consider the relaxed problem
\begin{equation}\label{eq:bidu}
    \min_{u} \calF(u) + \calG^{**}(u),
\end{equation}
where $\calG^{**}$ is the biconjugate of $\calG$, which is always convex. Existence and optimality conditions for the relaxed problem can readily be obtained. 
However, as we shall see, these optimality conditions are not directly amenable to numerical solution by Newton-type techniques. For this reason we consider a regularized optimality system
\begin{equation}\label{eq:regopt}
    \left\{\begin{aligned}
            - p_\gamma &\in \partial\calF(u_\gamma),\\
            u_\gamma &= (\partial\calG^*)_\gamma(p_\gamma),
    \end{aligned}\right.
\end{equation}
where $(\partial\calG^*)_\gamma$ is the Moreau--Yosida approximation of the subdifferential of the Fenchel conjugate $\calG^{*}$. Thus for the numerical realization, only $(\partial\calG^*)_\gamma$ is needed which can be computed without explicit knowledge of $\calG^{**}$. 
For problem \eqref{eq:problem_motiv}, the first relation of \eqref{eq:regopt} coincides with the usual state and adjoint equations, while the second relation allows a pointwise characterization; see \eqref{eq:h_gamma} below.

The remainder of this work is organized as follows.
In \cref{sec:convex}, we shall provide the abstract existence results, derive optimality conditions, and prove the convergence of solutions to system~\eqref{eq:regopt} to minimizers of problem~\eqref{eq:bidu}. 
\Cref{sec:pointwise} is dedicated to giving an explicit pointwise characterization of the subdifferential $\partial\calG^*$ and its Moreau--Yosida $(\partial\calG^*)_\gamma$ in the concrete case of switching control; two other functionals involving $|\cdot|_0$ (sparsity and multi-bang penalties) are discussed in \cref{sec:binary}. 
These characterizations allow addressing the significant questions related to the relaxation~\eqref{eq:bidu} of~\eqref{eq:orig} in \cref{sec:optimality}: We clarify the relation between the value of the costs in~\eqref{eq:bidu} and in~\eqref{eq:orig} in terms of the duality gap between $\calG$ and $\calG^*$, and show that in certain cases it can be guaranteed to be zero. If this is the case, then the solution to problem~\eqref{eq:bidu} is also a solution to problem~\eqref{eq:orig}. Moreover, we analyze to which extent the choice of the functional $(v_1,v_2) \mapsto |v_1 v_2|_0$, when used as part of control costs, in fact leads to optimal solutions of switching type. We shall be able to give a sufficient condition on the relation of $\alpha$ and $\beta$ for~\eqref{eq:bidu} that rule out free arcs, where $|v_1|$ and $|v_2|$ are both strictly positive but not equal, whereas singular arcs, on which $|v_1|=|v_2|>0$, may remain. 
\cref{sec:switching:solution} is concerned with the numerical solution of \eqref{eq:regopt} via a path-following semismooth Newton method. To guarantee convergence, a globalization is required. This guarantees superlinear convergence of the semismooth Newton algorithm in spite of the challenging cost, which combines continuous and discrete objectives. 
Finally, \cref{sec:examples} contains numerical tests for switching controls in the context of an elliptic and a parabolic partial differential equation. 

Let us put our work into perspective with respect to the existing literature. Casting the problem of switching controls as a nonconvex optimization problem involving the binary functional $|\cdot|_0$ is certainly new. Concerning the convex relaxation of nonconvex problems, we can draw from existing works. We only mention the monograph \cite{Ekeland:1999a}, where, however, the focus is on obtaining existence rather than on explicit optimality conditions and numerical realization.
The partial (Moreau--Yosida) regularization of nonsmooth convex finite-dimensional problems for the purpose of efficiently applying first-order methods was investigated in \cite{Beck:2012}.
Switching control has been studied mainly for ordinary differential equations; here we refer to \cite{Shorten:2007} for a survey with emphasis on stability of switching systems. The Hamilton--Jacobi--Bellman equation for switching controls was extensively studied in \cite{Dolcetta:1984} and \cite{Yong:1989}.
Switching control in the context of partial differential equations was especially investigated with respect to their improved flexibility over nonswitching controls for  stabilization \cite{Gugat:2008,Martinez:2002}. Controllability for systems with switching controls were studied in \cite{Zuazua:2011,Lu:2013}.
The hybrid nature of continuous and discrete phenomena when
the system switches among different modes is the focus of the work in 
\cite{Hante:2009,Hante:2013}. In \cite{Hante:2013} a relaxation 
technique combined with rounding strategies is proposed to solve 
mixed-integer programming problems arising in optimal control of partial 
differential equations. It is verified that the solution of the relaxed 
problems can be approximated with arbitrary accuracy by a solution 
satisfying the integer requirements. In \cite{Iftime:2009} optimal 
control of linear switched systems are considered, and an algorithmic 
treatment is proposed that relies on an exhaustive search which 
involves solving on the order of $m^k$ differential Riccati equations, 
where $m$ denotes the number of possible controller configurations and 
$k$ the number of predefined switching times.

\section{Convex relaxation and regularization approach}\label{sec:convex}

In this section we introduce  the abstract framework and recall relevant concepts from convex analysis. Consider the variational problem
\begin{equation}\label{eq:formal_prob}
    \min_{u\in U} \calJ(u) = \min_{u\in U} \calF(u) + \calG(u),
\tag{$\mathcal{P}$}
\end{equation}
where $U$ is a Hilbert space and $\calF:U\to\R$ is convex. If moreover $\calG:U\to\R\cup\{\infty\}$ is convex, any minimizer $\bar u \in U$
satisfies (under a regularity assumption stated below) the following necessary optimality conditions: There exists a $\bar p\in -\partial\calF(\bar u)\subset U^*$ such that $\bar p\in \partial\calG(\bar u)\subset U^*$, which holds if and only if $\bar u\in\partial\calG^*(\bar p)$; see, e.g., \cite[Proposition 4.4.4]{Schirotzek:2007}. Here,
\begin{equation}
    \calG^*:U^* \to \R\cup\{\infty\},\qquad  \calG^*(p) = \sup_{u\in U}\, \langle u,p \rangle - \calG(u),
\end{equation}
denotes the Fenchel conjugate of the convex functional $\calG$, and $\partial\calG^*$ denotes its convex subdifferential. (In the following, we identify the Hilbert space $U$ with its dual $U^*$ and consider $\calG^*:U \to \R\cup\{\infty\}$.) We thus obtain the primal-dual optimality system
\begin{equation}\label{eq:formal_opt}
    \left\{\begin{aligned}
            -\bar p &\in \partial\calF(\bar u),\\
            \bar u &\in \partial\calG^*(\bar p),
    \end{aligned}\right.
\end{equation}
which is well-defined even for nonconvex $\calG:U\to\R\cup\{\infty\}$ as in the situation we are interested in. To argue existence of a solution, we will show that the system~\eqref{eq:formal_opt} is the necessary optimality condition for
\begin{equation}\label{eq:formal_biconj}
    \min_{u\in U} \calF(u) + \calG^{**}(u),
\end{equation}
where $\calG^{**}=(\calG^*)^*$ is the biconjugate of $\calG$, and make the following standard assumptions:
\begin{equation}\label{eq:assumption_f}
    \left\{
        \begin{aligned}
            &\calF \text{ is convex and weakly lower-semicontinuous,}\\
            &\calG \text{ is proper and non-negative,}\\
            &\calF+\calG^{**} \text{ is radially unbounded.}
        \end{aligned}
    \right.
    \tag{\textsc{a}1}
\end{equation}
\begin{proposition}\label{thm:existence}
    Under assumption \eqref{eq:assumption_f}, the system~\eqref{eq:formal_opt} admits a solution $(\bar u,\bar p)\in U\times U$. If $\calF$ is strictly convex, this solution is unique.
\end{proposition}
\begin{proof}
    By assumption, $\calG:U\to \R+\cup\{\infty\}$ is bounded from below by $0$, which implies that $\calG^{**}\geq 0$ as well, see, e.g. \cite[Proposition 13.14]{Bauschke:2011}. Furthermore, Fenchel conjugates are always lower semicontinuous and convex, see, e.g. \cite[Proposition 13.11]{Bauschke:2011}. Together with assumption \eqref{eq:assumption_f} this implies that $\calF +\calG^{**}$ is convex, weakly lower semicontinuous, and radially unbounded, and thus a standard subsequence argument yields existence of a minimizer $\bar u\in U$ to \eqref{eq:formal_biconj}.

    Since $\dom\calF = U$ ensures that the stability condition
    \begin{equation}
        \bigcup_{\lambda\geq 0} \lambda (\dom\calF - \dom\calG^{**} ) \text{ is a closed vector space}
    \end{equation}
    holds, we can apply the sum rule for the convex subdifferential from \cite{Brezis:1986} and again appeal to \cite[Proposition 4.4.4]{Schirotzek:2007} for $\partial\calG^{**}$ to arrive at the necessary optimality conditions~\eqref{eq:formal_opt}.
\end{proof}

Problem \eqref{eq:formal_biconj} can be seen a convex relaxation of problem~\eqref{eq:formal_prob}.
This approach is thus related to the $\Gamma$-regularization in the calculus of variations, see, e.g., \cite[Chapter IX]{Ekeland:1999a}, although here we consider a more specific relaxation and pass to the biconjugate only in the nonconvex term rather than to the full biconjugate functional $\calJ^{**}$, which allows us to obtain explicit optimality conditions in the primal-dual  form \eqref{eq:formal_opt} that are useful for numerical computations.

In general, a solution to system~\eqref{eq:formal_opt} is not necessarily a minimizer of~\eqref{eq:formal_prob}, since for nonconvex $\calG$ we cannot rely on equality in the Fenchel--Young inequality (which requires the characterization of the convex subdifferential). In fact, a solution to problem~\eqref{eq:formal_prob} may not even exist. However, for the class of penalties we are interested in, it is possible to show that a solution to system~\eqref{eq:formal_opt} is \emph{suboptimal} in the sense that the corresponding functional value is within a certain distance of the infimum. 
This distance is given by the \emph{duality gap}
\begin{equation}\label{eq:dual_gap}
    \delta(u,p):=\calG(u) + \calG^*(p) - \scalprod{p,u}
\end{equation}
between $\calG$ and its Fenchel dual $\calG^*$. This gap is always non-negative by the Fenchel--Young inequality, and vanishes if $\calG$ is convex and $p\in\partial\calG(u)$.
\begin{lemma}\label{lem:dual_gap}
    Let $\calF$ satisfy \eqref{eq:assumption_f}, and let $(\bar u,\bar p)$ satisfy~\eqref{eq:formal_opt}. Then
    \begin{equation}\label{eq:dual_gap_j}
        \calJ(\bar u) \leq \calJ(u) + \delta(\bar u,\bar p) \quad\text{for all } u \in U.
    \end{equation}
\end{lemma}
\begin{proof}
    Assume that $(\bar u,\bar p)$ is a solution to system~\eqref{eq:formal_opt} and let $u\in U$ be arbitrary. Recall that the first relation of~\eqref{eq:formal_opt} then implies that
    \begin{equation}\label{eq:dual_gap_f}
        \calF(u) - \calF(\bar u) - \scalprod{-\bar p,u-\bar u} \geq 0.
    \end{equation}
    Furthermore, by definition~\eqref{eq:dual_gap} and the Fenchel--Young inequality (which holds for any proper $\calG$) we have that
    \begin{equation}
        \calG(u) - \calG(\bar u) -  \scalprod{\bar p,u-\bar u} = \calG(u) -  \scalprod{\bar p,u} + \calG^*(\bar p) -\delta(\bar u,\bar p) \geq -\delta(\bar u,\bar p).
    \end{equation}
    Hence,
    \begin{equation}
        \begin{split}
            \begin{aligned}[b]
                \calJ(u)-\calJ(\bar u) &=  (\calF(u)+\calG(u)) - (\calF(\bar u) + \calG(\bar u))\\
                                       &=
                (\calF(u)-\calF(\bar u) -\scalprod{-\bar p,u-\bar u}) + (\calG(u)-\calG(\bar u) -\scalprod{\bar p,u-\bar u}) \\
                &\geq  -\delta(\bar u,\bar p).
            \end{aligned}
            \qedhere
        \end{split}
    \end{equation}
\end{proof}

\bigskip

Since the subdifferential $\partial\calG^*$ is in general multivalued and not Lipschitz continuous, system~\eqref{eq:formal_opt} is not amenable to numerical solution. We therefore introduce the \emph{Moreau--Yosida regularization} of $\partial\calG^*$:
\begin{equation}\label{eq:yosida}
    u = (\partial\calG^*)_\gamma(p) := \frac1\gamma\left(p-\prox_{\gamma\calG^*}(p)\right),
\end{equation}
where
\begin{equation}
    \prox_{\gamma f}(v) = \arg\min_{w} f(w) + \frac1{2\gamma}\norm{w-v}^2
\end{equation}
is the \emph{proximal mapping} of $f$; see \cite{Moreau:1965}.
We recall the following properties of $\prox_{\gamma f}$ and $(\partial f)_\gamma$,  e.g., from \cite[Props.~12.29, 12.15, 23.10, 23.43, 12.9, 16.34]{Bauschke:2011}; see also \cite[Chapter 4.4]{Kunisch:2008a}.

\begin{proposition}\label{thm:convex:my}
    Let $f:H\to\R\cup\{\infty\}$ be a proper convex function on a Hilbert space $H$. Then,
    \begin{enumerate}[(i)]
        \item $(\partial f)_\gamma = (f_\gamma)'$, where
            \begin{equation}
                f_\gamma(v) =  f(\prox_{\gamma f}(v)) + \frac1{2\gamma} \norm{\prox_{\gamma f}(v)-v}^2
            \end{equation}
            is the \emph{Moreau-envelope} of $f$, which is real-valued and convex.
        \item $(\partial f)_\gamma$ is single-valued, maximally monotone and Lipschitz-continuous with constant $\gamma^{-1}$,
        \item $\norm{(\partial f)_\gamma(v)}_H \leq \inf_{q\in\partial f(v)}\norm{q}_H$ for all $v\in H$,
        \item $  f\left(\prox_{\gamma f}(v)\right) \le f_\gamma(v) \le    f(v)$ for all $\gamma >0$ and $v\in H$,
        \item $\prox_{\gamma f} = (\Id + \gamma\partial f)^{-1}$ (the \emph{resolvent} of $\partial f$).
    \end{enumerate}
\end{proposition}
From the last property, we can see that
\begin{equation}
    (\partial f)_\gamma = \frac1\gamma\left(\Id - (\Id + \gamma\partial f)^{-1}\right) = \partial f \circ (\Id + \gamma\partial f)^{-1},
\end{equation}
i.e., $ (\partial f)_\gamma$ is indeed the Moreau--Yosida regularization of $ \partial f$.

For brevity, we set $\calG_\gamma^*:=(\calG^*)_\gamma$ and $H_\gamma :=(\partial\calG^*)_{\gamma}$ from here on and consider 
the regularized optimality system
\begin{equation}\label{eq:formal_opt_reg}
    \left\{\begin{aligned}
            - p_\gamma &\in \partial\calF(u_\gamma),\\
            u_\gamma &= H_\gamma(p_\gamma).
    \end{aligned}\right.
\end{equation}
Arguing as in \cref{thm:existence}, existence of a solution follows from the fact that this system is the necessary optimality condition for the problem
\begin{equation}\label{eq:formal_prob_reg}
    \min_u \calF(u) + (\calG^*_\gamma)^*(u),
\end{equation}
using that $\calG^*_\gamma \leq \calG^*$ implies that $0\leq \calG^{**}\leq (\calG^*_\gamma)^*$ and that $H_\gamma = (\partial\calG^*)_\gamma$ is single-valued by \cref{thm:convex:my}\,(i,ii).
\begin{proposition}\label{thm:existence_reg}
    Under assumption \eqref{eq:assumption_f}, the system~\eqref{eq:formal_opt_reg} admits a solution $(u_\gamma,p_\gamma)\in U\times U$. If $\calF$ is strictly convex, this solution is unique.
\end{proposition}

The convergence $(u_\gamma,p_\gamma)\to(\bar u,\bar p)$ as $\gamma\to 0$ requires additional assumptions on $\calF$ and $\calG$:
\begin{align}
    &\left\{\begin{aligned}
&\text{(i) } \calF \text{ is Fréchet differentiable, $\calF'$ has weakly closed graph, and}\\
    &\text{(ii) }\{\calF(u_\gamma)\}_{\gamma>0} \text{ bounded implies }  \{{\calF'(u_\gamma)}\}_{\gamma>0} \text{ bounded,}
        \end{aligned}\right.
        \label{eq:assumption_f_bd} \tag{\textsc{a}2} \\[0.5ex]
        &\{p_\gamma\}_{\gamma>0} \text{ bounded implies }  \big\{\inf_{q\in\partial\calG^*(p_\gamma)}\norm{q}_U\big\}_{\gamma>0} \text{ bounded.}
        \label{eq:assumption_g_bd} \tag{\textsc{a}3}
    \end{align}
    We point out that (\ref{eq:assumption_f_bd}\,ii) is generically satisfied for functionals of the type $\calF(u) = F(S(u))$, where
    \begin{enumerate}[(i)]
        \item $F:Y\to\R$ is radially unbounded on a Banach space $Y$,
        \item $F$ is Fr\'echet differentiable and $F'$ is bounded on bounded sets,
        \item $S:U\to Y$ is Fr\'echet differentiable and $S'(u)^*$ is uniformly bounded on $U$,
    \end{enumerate}
    since in this case boundedness of $\calF(u_\gamma)$ implies boundedness of $y_\gamma:=S(u_\gamma)$ and hence boundedness of $\calF'(u_\gamma) = S'(u_\gamma)^*F'(y_\gamma)$.
    In particular, it holds for many common tracking-type functionals of the form $F(y)=\frac12\norm{y-z}_Y^2$ and bounded linear control-to-state mappings $S$. In this case, $\calF'(u) = S^*(Su-z)$ and  (\ref{eq:assumption_f_bd}\,i) trivially holds.
    Assumption \eqref{eq:assumption_g_bd} is more restrictive but satisfied for the class of functionals we shall consider later on.
    \begin{proposition}\label{thm:convergence}
        If $\calF$ and $\calG$ satisfy assumptions~\eqref{eq:assumption_f}--\eqref{eq:assumption_g_bd}, the family $\{(u_\gamma,p_\gamma)\}_{\gamma>0}$ contains a subsequence converging weakly as $\gamma\to 0$ to a solution $(\bar u,\bar p)$ to system~\eqref{eq:formal_opt}. If $\calF$ is strictly convex, the whole sequence converges weakly.
    \end{proposition}
    \begin{proof}
        First, observe that
        \begin{equation}
            (\calG^*_\gamma)^*(0) = \sup_{p\in U} -\calG^*_\gamma(p) = \inf_{p\in U} \calG^*_\gamma(p) \leq \inf_{p\in U} \calG^*(p)
        \end{equation}
        by \cref{thm:convex:my}\,(iii). By the optimality of $u_\gamma$ we thus have for any $\gamma>0$ that
        \begin{equation}
            \calF(u_\gamma) \leq \calF(u_\gamma) + (\calG^*_\gamma)^*(u_\gamma) \leq \calF(0) + \inf_{p\in U} \calG^*(p).
        \end{equation}
        Hence, $\{\calF(u_\gamma)\}_{\gamma>0}$ is bounded, and assumption~\eqref{eq:assumption_f_bd}  yields that
        \begin{equation}
            \{p_\gamma\}_{\gamma>0} =   \{-\calF'(u_\gamma)\}_{\gamma>0}
        \end{equation}
        is bounded. From assumption~\eqref{eq:assumption_g_bd} together with \cref{thm:convex:my}\,(iii) it then follows that for every $\gamma>0$, we have that
        \begin{equation}
            \norm{u_\gamma}_U = \norm{H_\gamma(p_\gamma)}_U \leq \inf_{q\in \partial\calG^*(p_\gamma)}\norm{q}_U \leq C,
        \end{equation}
        i.e., $\{H_\gamma(p_\gamma)\}_{\gamma>0}$ and $\{u_\gamma\}_{\gamma>0}$ are bounded. Hence, there exist subsequences $\{u_{\gamma_n}\}_{n\in\N}$, $\{p_{\gamma_n}\}_{n\in\N}$ and $\{H_{\gamma_n}(p_{\gamma_n})\}_{n\in\N}$ converging weakly in $U$ to some $\hat u$, $\hat p$, and $\hat y$, respectively. The weak closedness of $\calF'$ then yields
        \begin{equation}
            \hat p = -\calF'(\hat u).
        \end{equation}
        For the second relation of system~\eqref{eq:formal_opt}, we first observe that due to the monotonicity of $\calF'$ and using both relations of system~\eqref{eq:formal_opt_reg}, we have for any $\gamma_1,\gamma_2>0$ that
        \begin{equation}
            \scalprod{H_{\gamma_1}(p_{\gamma_1}) - H_{\gamma_2}(p_{\gamma_2}),p_{\gamma_1}-p_{\gamma_2}} 
            = -\scalprod{u_{\gamma_1}-u_{\gamma_2},\calF'(u_{\gamma_1})-\calF'(u_{\gamma_2})}
            \leq 0,
        \end{equation}
        and hence that for any sequence $\{\gamma_n\}_{n\in\N}$ with $\gamma_n\to 0$,
        \begin{equation}
            \limsup_{n,m\to\infty} \ \scalprod{H_{\gamma_n}(p_{\gamma_n}) - H_{\gamma_m}(p_{\gamma_m}),p_{\gamma_n}-p_{\gamma_m}}\leq 0.
        \end{equation}
        Since $H_\gamma$ is monotone, we can apply
        \cite[Lemma 1.3(e)]{Brezis:1970} to obtain that $\hat u = \partial\calG^*(\hat p)$, i.e., $(\hat u,\hat p)$ satisfies system~\eqref{eq:formal_opt}.

        If $\calF$ is strictly convex, the solution to system~\eqref{eq:formal_opt} is unique, and the claim follows from a subsequence--subsequence argument.
    \end{proof}

    To conclude this section, we compare the Moreau--Yosida regularization with the following complementarity formulation of the second relation of system~\eqref{eq:formal_opt}: For any $\gamma>0$, we have that
    \begin{equation}
        \begin{aligned}
            u \in \partial\calG^*(p) &\Leftrightarrow p + \gamma u \in (\Id + \gamma\partial\calG^*)(p)\\
                                     &\Leftrightarrow p \in (\Id + \gamma\partial\calG^*)^{-1}(p+\gamma u)\\
                                     &\Leftrightarrow p = \prox_{\gamma\calG^*}(p+\gamma u)\\
                                     &\Leftrightarrow u 
            = \frac1\gamma\left((p+\gamma u)-\prox_{\gamma\calG^*}(p+\gamma u)\right)
            = (\partial\calG^*)_\gamma(p+\gamma u) = (\calG^*_\gamma)'(p+\gamma u),
        \end{aligned}
    \end{equation}
    see also \cite[Theorem 4.41]{Kunisch:2008a}. The subdifferential inclusion can thus be equivalently expressed as a nonlinear equation. While the subdifferential inclusion is explicit with respect to $u$, the nonlinear equation is implicit. Moreover, the appearance of $u$ in the proximal mapping rules out the effective use of
    semismooth Newton methods for the applications we have in mind. On the other hand, note that the Moreau--Yosida approximation~\eqref{eq:yosida} differs only in the absence of $\gamma u$ on the right hand side of the last equality. Hence semismooth Newton methods will be applicable.

    \section{Switching cost functional \texorpdfstring{$\scriptstyle g$}{g}}\label{sec:pointwise}

    To make practical use of the proposed approach, we require an explicit, pointwise, characterization of $\partial\calG^*$ and $(\partial\calG^*)_\gamma$. For this, we exploit the integral nature of functionals of the~type
    \begin{equation}
        \calG (u) = \int_D g(u(x)) \,dx
    \end{equation}
    with $D\subset  \R^d,$ for some $d \ge 1$, which allows computing the Fenchel conjugate and its subdifferential pointwise as well; see, e.g., \cite[Props.~IV.1.2, IX.2.1]{Ekeland:1999a}, \cite[Prop.~16.50]{Bauschke:2011}.

    Specifically, we consider here the switching cost functional on $\R^2$,
    \begin{equation}\label{eq:pointwise:switching}
        g(v) = \frac\alpha2(v_1^2 + v_2^2) + \beta |v_1 v_2|_0\,.
    \end{equation}
    Other penalties of this class are discussed in \cref{sec:binary}.
    The use of the term $|v_1v_2|_0$ enhances switching between the control variables $v_1$ and $v_2$ in such a manner that simultaneous nontriviality of both of them is penalized.  We shall give sufficient conditions which guarantee that in fact $v_1$ and $v_2$ are not simultaneously nontrivial except for a singular set of controls for which $|v_1| =|v_2| \leq \sqrt{2\beta/\alpha}$.

    \subsection{Fenchel conjugate of \texorpdfstring{$\scriptstyle g$}{g}}\label{sec:switching:conjugate}

    To characterize 
    \begin{equation}\label{eq:switching:fenchel_pt}
        g^*(q) = \sup_{v\in\R^2}v\cdot q-g(v),
    \end{equation}
    first note that the function $v\mapsto g(v) - v\cdot q$ is lower semicontinuous 
    and radially unbounded. The supremum in \eqref{eq:switching:fenchel_pt} is thus attained at some $\bar v\in \R^2$. We then discriminate the following~cases:
    \begin{enumerate}[(i)]
        \item $\bar v_1= 0$, in which case $g(\bar v) = \frac\alpha2 \bar v_2^2$.
            The supremum in~\eqref{eq:switching:fenchel_pt} is attained if and only if the necessary optimality condition $q_2 - \alpha \bar v_2=0$ holds. Solving for $\bar v_2$ and inserting into~\eqref{eq:switching:fenchel_pt} yields
            \begin{equation}
                g^*(q) = \frac1{2\alpha} q_2^2.
            \end{equation}
        \item $\bar v_2 = 0$, in which case $g(\bar v) = \frac\alpha2 \bar v_1^2$.
            By the same argument as in case (i) we obtain
            \begin{equation}
                g^*(q) = \frac1{2\alpha} q_1^2.
            \end{equation}
        \item $\bar v_1,\bar v_2\neq 0$, in which case $g(\bar v) = \frac\alpha2 (\bar v_1^2+\bar v_2^2)+\beta$.
            Again, using the necessary optimality condition for the supremum in~\eqref{eq:switching:fenchel_pt} yields
            \begin{equation}
                g^*(q) = \frac1{2\alpha} (q_1^2+q_2^2)-\beta.
            \end{equation}
    \end{enumerate}
    It remains to decide which of these cases is attained based on the value of $q$. For this purpose,~define
    \begin{equation}
        g_i^*(q) = 
        \begin{cases}
            \frac1{2\alpha} q_i^2 & \text{if } i \in\{1,2\},\\
            \frac1{2\alpha} (q_1^2+q_2^2)-\beta & \text{if } i = 0.
        \end{cases}
    \end{equation}
    Since all $g_i^*$ are finite, the supremum in~\eqref{eq:switching:fenchel_pt} is attained at
    \begin{equation}
        g^*(q) = \max_{i\in\{0,1,2\}} g_i^*(q).
    \end{equation}
    From the definition, we have that $g_1^*(q) \geq g_2^*(q)$ if $|\bar v_1| \geq |\bar v_2|$ and $g_1^*(q) \geq g_0^*(q)$ if $|\bar v_2|\leq \sqrt{2\alpha\beta}$; similarly for $g_2^*(q)$. Conversely, $g_0^*(q) \geq g_i^*(q)$ if $|\bar v_j|\leq \sqrt{2\alpha\beta}$, $j=1,2$.
    Summarizing the above, we have
    \begin{equation+}\label{eq:switching:fenchel}
        g^*(q) =  
        \begin{cases}
            \frac1{2\alpha} q_1^2 & \text{if } |q_1|\geq |q_2| \text{ and }|q_2|\leq \sqrt{2\alpha\beta},\\
            \frac1{2\alpha} q_2^2 & \text{if } |q_1|\leq |q_2| \text{ and }|q_1|\leq \sqrt{2\alpha\beta},\\
            \frac1{2\alpha} (q_1^2+q_2^2)-\beta & \text{if } |q_1|,|q_2| \geq \sqrt{2\alpha\beta}.
        \end{cases}
    \end{equation+}

    \subsection{Subdifferential of \texorpdfstring{$\scriptstyle g^*$}{g}}\label{sec:switching:subdifferential}

    Since $g^*$ is the maximum of a finite number of convex functions, its subdifferential is given by 
    \begin{equation}
        \partial g^*(q) = \overline{\mathrm{co}} \left(\bigcup_{\{i:g^*(q)= g_i^*(q)\}}\left\{ (g_{i}^*)'(q)\right\}\right),
    \end{equation}
    where $\overline{\mathrm{co}}$ denotes the closed convex hull; see, e.g., \cite[Corollary 4.3.2]{Hiriart:2001}. 
    We make a case distinction based on all possibilities for $g^*(q)=g_i^*(q)$, $i\in\{0,1,2\}$:
    \begin{enumerate}[(i)]
        \item $g^*(q) = g_1^*(q)$ only, which is the case if and only if
            \begin{equation}
                q\in Q_1 := \set{q\in\R^2}{|q_1| > |q_2| \text{ and } |q_2|< \sqrt{2\alpha\beta}}.
            \end{equation}
            Here the subdifferential is single-valued and given by
            \begin{equation}
                \partial g^*(q) = \left(\left\{\tfrac1\alpha q_1\right\},\left\{0\right\}\right).
            \end{equation}
        \item $g^*(q) = g_2^*(q)$ only, which is the case if and only if
            \begin{equation}
                q\in Q_2 := \set{q\in\R^2}{|q_2| > |q_1| \text{ and } |q_1|< \sqrt{2\alpha\beta}}.
            \end{equation}
            Here,
            \begin{equation}
                \partial g^*(q) = \left(\left\{0\right\},\left\{\tfrac1\alpha q_2\right\}\right).
            \end{equation}
        \item $g^*(q) = g_0^*(q)$ only, which is the case if and only if
            \begin{equation}
                q\in Q_0 := \set{q\in\R^2}{|q_1|,|q_2| >  \sqrt{2\alpha\beta}}.
            \end{equation}
            Here,
            \begin{equation}
                \partial g^*(q) = \left(\left\{\tfrac1\alpha q_1\right\},\left\{\tfrac1\alpha q_2\right\}\right).
            \end{equation}
        \item $g^*(q) = g_1^*(q) = g_0^*(q)\neq g_2^*(q)$, which is the case if and only if 
            \begin{equation}
                q\in Q_{10} := \set{q\in\R^2}{|q_1|> |q_2| = \sqrt{2\alpha\beta}}.
            \end{equation}
            Here, the subdifferential is given by the convex hull of $\{(g_1^*)'(q),(g_0^*)'(q)\}$, i.e.,
            \begin{equation}
                \partial g^*(q) = \left(\left\{\tfrac1\alpha q_1\right\},\left[0,\tfrac1\alpha q_2\right]\right).
            \end{equation}
            To keep the notation concise, we use the convention $[a,b]:=[\min\{a,b\},\max\{a,b\}]$ here and below.
        \item $g^*(q) = g_2^*(q) = g_0^*(q)\neq g_1^*(q)$, which is the case if and only if
            \begin{equation}
                q\in Q_{20} := \set{q\in\R^2}{|q_2|> |q_1| = \sqrt{2\alpha\beta}}.
            \end{equation}
            Here, 
            \begin{equation}
                \partial g^*(q) = \left(\left[0,\tfrac1\alpha q_1\right],\left\{\tfrac1\alpha q_2\right\}\right).
            \end{equation}
        \item  $g^*(q) = g_1^*(q) = g_2^*(q)$, which is the case if and only if 
            \begin{equation}
                q\in Q_{12} := \set{q\in\R^2}{|q_1| = |q_2| \leq \sqrt{2\alpha\beta}}.
            \end{equation}
            Here,
            \begin{equation}
                \partial g^*(q) = \set{\left(\tfrac{t}\alpha q_1,\tfrac{1-t}\alpha q_2\right)}{t\in[0,1]}.
            \end{equation}
            Note that this also includes the case $g^*(q) = g_1^*(q) = g_2^*(q)= g_0^*(q)$, since then $(g_0^*)'(q)\in  \partial g^*(q)$.
    \end{enumerate}
    Since $\R^2$ is the disjoint union of the sets $Q_i$ defined above, see \cref{fig:gconj}, we thus obtain a complete characterization of the subdifferential $\partial g^*(q)$. 
    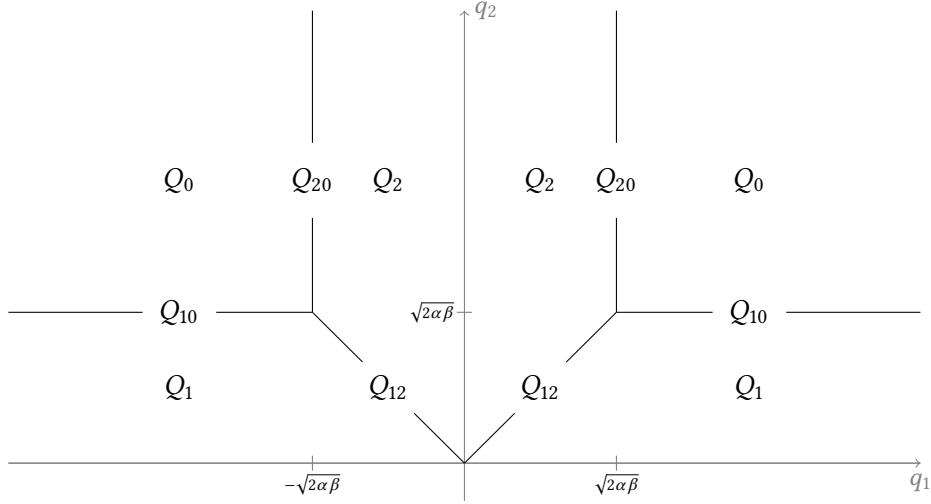
\begin{figure}
        \centering
        \begin{tikzpicture}[]
    \draw[gray,->](0,-2.5) -- (0,4) node[right] {\small $q_2$};
    \draw[gray,->](-6,-2) -- (6,-2) node[below] {\small $q_1$};
    \draw[gray](-6,-2) -- (6,-2);
    \draw[gray](2,-1.9) -- (2,-2.1);
    \draw[gray](-2,-1.9) -- (-2,-2.1);
    \draw[gray](-0.1,0) -- (0.1,0);
    \draw(-2,0) -- (0,-2);
    \draw(2,0) -- (0,-2);
    \draw(-6,0) -- (-2,0);
    \draw(2,0) -- (6,0);
    \draw(-2,0) -- (-2,4);
    \draw(2,0) -- (2,4);
    \draw (-3.75,1.75) node {$Q_0$};
    \draw (3.75,1.75) node {$Q_0$};
    \draw (-1,1.75) node {$Q_{2}$};
    \draw (1,1.75) node {$Q_{2}$};
    \draw (3.75,-1) node {$Q_{1}$};
    \draw (-3.75,-1) node {$Q_{1}$};

    \draw (-1,-1) node[circle,fill=white] {$Q_{12}$};
    \draw (1,-1) node[circle,fill=white] {$Q_{12}$};
    \draw (-3.75,0) node[circle,fill=white] {$Q_{10}$};
    \draw (3.75,0) node[circle,fill=white] {$Q_{10}$};
    \draw (-2,1.75) node[circle,fill=white] {$Q_{20}$};
    \draw (2,1.75) node[circle,fill=white] {$Q_{20}$};

    \draw (2,-2) node[below] {\tiny $\sqrt{2\alpha\beta}$};
    \draw (-2,-2) node[below] {\tiny $-\sqrt{2\alpha\beta}$};
    \draw (0,0) node[left] {\tiny $\sqrt{2\alpha\beta}$};

\end{tikzpicture}
        \caption{Subdomains $Q_i\subset\R^2$ for the definition of $\partial g^*$.}
        \label{fig:gconj}
    \end{figure}

    \subsection{Proximal mapping of \texorpdfstring{$\scriptstyle g^*$}{g}}\label{sec:switching:resolvent}

    For the Moreau--Yosida regularization or the complementarity formulation, we need to compute the proximal mapping of $g^*$ or, equivalently, the resolvent of $\partial g^*$.
    For given $\gamma>0$ and $v\in\R$, the resolvent $w:=(\mathrm{Id}+\gamma\partial g^*)^{-1}(v)$ is characterized by the subdifferential inclusion
    \begin{equation}\label{eq:resolvent}
        v\in (\mathrm{Id}+\gamma\partial g^*)(w) = \{w\}+\gamma\partial g^*(w).
    \end{equation}
    Note that this implies 
    \begin{equation}\label{eq:subdiff_inclusion}
        v \in \left[w,(1+\tfrac\gamma\alpha)w\right]\qquad\text{or equivalently that }\qquad w \in \left[\tfrac{\alpha}{\alpha+\gamma}v,v\right],
    \end{equation}
    and hence that $\sign(v_j) = \sign(w_j)$, $j=1,2$.
    We now follow the case discrimination in the characterization of the subdifferential.
    \begin{enumerate}[(i)]
        \item $w\in Q_1$:  In this case, the subdifferential inclusion \eqref{eq:resolvent}  yields $v_1 = (1+\frac\gamma\alpha)w_1$ and $v_2 = w_2$; solving for $w_1,w_2$ and inserting the result into the definition of $Q_1$ yields
            \begin{equation}
                w = \left(\tfrac\alpha{\alpha+\gamma}v_1,v_2\right) \quad \text{and} \quad |v_1|>(1+\tfrac\gamma\alpha)|v_2|,\quad |v_2| < \sqrt{2\alpha\beta}.
            \end{equation}

        \item $w\in Q_2$:  In this case,  $v_1 = w_1$ and $v_2 = (1+\frac\gamma\alpha)w_2$, and as in case (i) we have that
            \begin{equation}
                w = \left(v_1,\tfrac\alpha{\alpha+\gamma}v_2\right) \quad \text{and} \quad  |v_2|>(1+\tfrac\gamma\alpha)|v_1|,\quad |v_1| < \sqrt{2\alpha\beta}.
            \end{equation}

        \item $w\in Q_0$:  In this case,  $v_1 = (1+\frac\gamma\alpha)w_1$ and $v_2 = (1+\frac\gamma\alpha)w_2$, and hence
            \begin{equation}
                w = \left(\tfrac\alpha{\alpha+\gamma}v_1,\tfrac\alpha{\alpha+\gamma}v_2\right) \quad \text{and} \quad  |v_1|>(1+\tfrac\gamma\alpha)\sqrt{2\alpha\beta},\quad |v_2| > (1+\tfrac\gamma\alpha)\sqrt{2\alpha\beta}.
            \end{equation}

        \item $w\in Q_{10}$: In this case, $v_1 = (1+\frac\gamma\alpha)w_1$ and $v_2 \in [w_2,(1+\frac\gamma\alpha)w_2]$. Since $\sign(w_2)=\sign(v_2)$, we have from the definition of $Q_{10}$ that $w_2 = \sign(v_2)\sqrt{2\alpha\beta}$. Hence 
            \begin{equation}
                w = \left(\tfrac\alpha{\alpha+\gamma}v_1,\sign(v_2) \sqrt{2\alpha\beta}\right) \quad \text{and} \quad \sqrt{2\alpha\beta}\leq |v_2| \leq (1+\tfrac\gamma\alpha)\sqrt{2\alpha\beta},\quad |v_1|>(1+\tfrac\gamma\alpha)\sqrt{2\alpha\beta}.
            \end{equation}

        \item $w\in Q_{20}$: In this case,  $v_2 = (1+\frac\gamma\alpha)w_2$ and $v_1 \in [w_1,(1+\frac\gamma\alpha)w_1]$. As in (iv), we have that 
            \begin{equation}
                w = \left(\sign(v_1) \sqrt{2\alpha\beta},\tfrac\alpha{\alpha+\gamma}v_2\right) \quad \text{and} \quad \sqrt{2\alpha\beta}\leq |v_1| \leq (1+\tfrac\gamma\alpha)\sqrt{2\alpha\beta},\quad |v_2|>(1+\tfrac\gamma\alpha)\sqrt{2\alpha\beta}.
            \end{equation}

        \item $w\in Q_{12}$: 
            In this case, $v_1\in[w_1,(1+\frac\gamma\alpha)w_1]$ and $v_2 \in [w_2,(1+\frac\gamma\alpha)w_2]$. This does not yield an explicit value for $w$, although the definition of $Q_{12}$ implies that $|w_1|=|w_2|\leq \sqrt{2\alpha\beta}$.
            We therefore turn to the equivalent characterization of $w$ via the proximal mapping
            \begin{equation}
                w = \prox_{\gamma g^*}(v) = \mathop\mathrm{argmin}_{|z_1|=|z_2|\leq \sqrt{2\alpha\beta}} \frac1{2\gamma}|z-v|_2^2+ g^*(z).
            \end{equation}
            First, assume that $z_1=z_2=:z$ (which implies $\sign(v_1)=\sign(z)=\sign(v_2)$). The minimizer of the reduced problem is then given by the projection of the unconstrained minimizer $z =\frac\alpha{2\alpha+\gamma}(v_1+v_2)$ to the (convex) feasible set $[-\sqrt{2\alpha\beta},\sqrt{2\alpha\beta}]$, i.e.,
            \begin{equation}
                w = 
                \begin{cases}                         
                    \left(\tfrac\alpha{2\alpha+\gamma}(v_1+v_2),\tfrac\alpha{2\alpha+\gamma}(v_1+v_2)\right) & \text{if } \tfrac\alpha{2\alpha+\gamma}|v_1+v_2|\leq \sqrt{2\alpha\beta},\\
                    \left(\sign(v_1)\sqrt{2\alpha\beta},\sign(v_2)\sqrt{2\alpha\beta}\right) &\text{if } \tfrac\alpha{2\alpha+\gamma}|v_1+v_2|> \sqrt{2\alpha\beta}.
                \end{cases}
            \end{equation}
            Inserting each of these values for $w$ into the relation $v\in[w,(1+\tfrac\gamma\alpha)w]$ yields (after some algebraic manipulations) 
            \begin{equation}
                \tfrac{\alpha}{\alpha+\gamma}|v_2| \leq |v_1| \leq (1+\tfrac\gamma\alpha)|v_2|
            \end{equation}
            and 
            \begin{equation}
                \sqrt{2\alpha\beta} \leq |v_1|,|v_2| \leq (1+\tfrac\gamma\alpha)\sqrt{2\alpha\beta},
            \end{equation}
            respectively.

            We argue similarly for $z_1 = - z_2$ (where $\sign(v_1)=\sign(z)=-\sign(v_2)$). Combining the two cases, we obtain 
            \begin{equation}
                \begin{multlined}
                    w = \left(\sign(v_1)\tfrac\alpha{2\alpha+\gamma}(|v_1|+|v_2|),\sign(v_2)\tfrac\alpha{2\alpha+\gamma}(|v_1|+|v_2|)\right) \\ \text{and}\quad   \tfrac{\alpha}{\alpha+\gamma}|v_2| \leq |v_1| \leq (1+\tfrac\gamma\alpha)|v_2|,\quad |v_1|+|v_2|\leq (2+\tfrac\gamma\alpha)\sqrt{2\alpha\beta},
                \end{multlined}
            \end{equation}
            and 
            \begin{equation}
                \begin{multlined}
                    w = \left(\sign(v_1)\sqrt{2\alpha\beta},\sign(v_2)\sqrt{2\alpha\beta}\right)  \\  \text{and}\quad \sqrt{2\alpha\beta} \leq |v_1|,|v_2| \leq (1+\tfrac\gamma\alpha)\sqrt{2\alpha\beta},\quad |v_1|+|v_2|>(2+\tfrac\gamma\alpha) \sqrt{2\alpha\beta}.
                \end{multlined}
            \end{equation}
    \end{enumerate}
    Inserting this into the definition of the Moreau--Yosida regularization
    \begin{equation}
        (\partial g^*)_\gamma(q) = \frac1\gamma\left(q-\prox_{\gamma g^*}(q)\right)
    \end{equation}
    and simplifying yields
    \begin{equation}\label{eq:h_gamma}
        (\partial g^*)_\gamma(q) =
        \begin{cases}
            \left(\frac1{\alpha+\gamma}q_1,0\right) & \text{if } q \in Q_1^\gamma,\\
            \left(0,\frac1{\alpha+\gamma}q_2\right) & \text{if } q \in Q_2^\gamma,\\
            \left(\frac1{\alpha+\gamma}q_1,\frac1{\alpha+\gamma}q_2\right) & \text{if }q \in Q_0^\gamma,\\
            \left(\frac1{\alpha+\gamma}q_1,\frac1\gamma\left(q_2-\sign(q_2)\sqrt{2\alpha\beta}\right)\right) & \text{if }q \in Q_{10}^\gamma,\\
            \left(\frac1\gamma\left(q_1-\sign(q_1)\sqrt{2\alpha\beta}\right),\frac1{\alpha+\gamma}q_2\right) & \text{if }q \in Q_{20}^\gamma,\\
            \left(\frac1\gamma\left(q_1-\sign(q_1)\sqrt{2\alpha\beta}\right),\frac1\gamma\left(q_2-\sign(q_2)\sqrt{2\alpha\beta}\right)\right) 
            & \text{if }q \in Q_{00}^\gamma,\\
            \left(\frac1\gamma\left(\tfrac{\alpha+\gamma}{2\alpha+\gamma}q_1-\sign(q_1)\tfrac{\alpha}{2\alpha+\gamma}|q_2|\right)\right.,\\
            \;\, \left.\frac1\gamma\left(\tfrac{\alpha+\gamma}{2\alpha+\gamma}q_2-\sign(q_2)\tfrac{\alpha}{2\alpha+\gamma}|q_1|\right)\right) 
            & \text{if }q \in Q_{12}^\gamma,\\
        \end{cases}
    \end{equation}
    where
    \begin{align}
        Q_1^\gamma &= \set{q}{|q_1|>(1+\tfrac\gamma\alpha)|q_2|\text{ and } |q_2| < \sqrt{2\alpha\beta}},\\
        Q_2^\gamma &= \set{q}{|q_2|>(1+\tfrac\gamma\alpha)|q_1|\text{ and } |q_1| < \sqrt{2\alpha\beta}},\\
        Q_0^\gamma &= \set{q}{|q_1|,|q_2|>(1+\tfrac\gamma\alpha)\sqrt{2\alpha\beta}},\\
        Q_{10}^\gamma &= \set{q}{|q_1|\in\left[\sqrt{2\alpha\beta},(1+\tfrac\gamma\alpha)\sqrt{2\alpha\beta}\right] \text{ and } |q_2|>(1+\tfrac\gamma\alpha)\sqrt{2\alpha\beta}},\\
        Q_{20}^\gamma &= \set{q}{|q_2|\in\left[\sqrt{2\alpha\beta},(1+\tfrac\gamma\alpha)\sqrt{2\alpha\beta}\right] \text{ and } |q_1|>(1+\tfrac\gamma\alpha)\sqrt{2\alpha\beta}},\\
        Q_{00}^\gamma &= \set{q}{|q_1|,|q_2|\in\left[\sqrt{2\alpha\beta},(1+\tfrac\gamma\alpha)\sqrt{2\alpha\beta}\right] \text{ and } |q_1|+|q_2|>(2+\tfrac\gamma\alpha)\sqrt{2\alpha\beta}},\\
        Q_{12}^\gamma &= \set{q}{|q_1|\in\left[\tfrac\alpha{\alpha+\gamma}|q_2|,(1+\tfrac\gamma\alpha)|q_2|\right] \text{ and }  |q_1|+|q_2| \leq (2+\tfrac\gamma\alpha)\sqrt{2\alpha\beta}},
    \end{align}
    see \cref{fig:greg}.

    \begin{figure}
        \centering
        \begin{tikzpicture}[]
    \draw[gray,->](0,-2.5) -- (0,4) node[right] {\small $q_2$};
    \draw[gray,->](-6,-2) -- (6,-2) node[below] {\small $q_1$};
    \draw[gray](-6,-2) -- (6,-2);
    \draw[gray](2,-1.9) -- (2,-2.1);
    \draw[gray](3.5,-1.9) -- (3.5,-2.1);
    \draw[gray](-2,-1.9) -- (-2,-2.1);
    \draw[gray](-3.5,-1.9) -- (-3.5,-2.1);
    \draw[gray](-0.1,0) -- (0.1,0);
    \draw[gray](-0.1,1.5) -- (0.1,1.5);
    \draw(-3.5,0) -- (0,-2);
    \draw(3.5,0) -- (0,-2);
    \draw(-2,1.5) -- (0,-2);
    \draw(2,1.5) -- (0,-2);
    \draw(-3.5,0) -- (-2,1.5);
    \draw(3.5,0) -- (2,1.5);
    \draw(-6,0) -- (-3.5,0);
    \draw(2,1.5) -- (6,1.5);
    \draw(-2,1.5) -- (-6,1.5);
    \draw(3.5,0) -- (6,0);
    \draw(-3.5,0) -- (-3.5,4);
    \draw(-2,1.5) -- (-2,4);
    \draw(3.5,0) -- (3.5,4);
    \draw(2,1.5) -- (2,4);
    \draw (-4.75,2.75) node {$Q_0^\gamma$};
    \draw (-2.75,2.75) node {$Q_{20}^\gamma$};
    \draw (-3,1) node {$Q_{00}^\gamma$};
    \draw (-2,-0) node{$Q_{12}^\gamma$};
    \draw (2,-0) node {$Q_{12}^\gamma$};
    \draw (-1,2.75) node {$Q_{2}^\gamma$};
    \draw (1,2.75) node {$Q_{2}^\gamma$};
    \draw (2.75,2.75) node {$Q_{20}^\gamma$};
    \draw (3,1) node {$Q_{00}^\gamma$};
    \draw (-4.75,0.75) node {$Q_{10}^\gamma$};
    \draw (4.75,0.75) node {$Q_{10}^\gamma$};
    \draw (4.75,2.75) node {$Q_{0}^\gamma$};
    \draw (4.75,-1) node {$Q_{1}^\gamma$};
    \draw (-4.75,-1) node {$Q_{1}^\gamma$};

    \draw (-2,-2) node[below] {\tiny $-\sqrt{2\alpha\beta}$};
    \draw (-3.5,-2) node[below] {\tiny $-\left(1+\tfrac\gamma\alpha\right)\sqrt{2\alpha\beta}$};
    \draw (2,-2) node[below] {\tiny $\sqrt{2\alpha\beta}$};
    \draw (3.5,-2) node[below] {\tiny $\left(1+\tfrac\gamma\alpha\right)\sqrt{2\alpha\beta}$};
    \draw (0,0) node[left] {\tiny $\sqrt{2\alpha\beta}$};
    \draw (0,1.5) node[left] {\tiny $\left(1+\tfrac\gamma\alpha\right)\sqrt{2\alpha\beta}$};

\end{tikzpicture}
        \caption{Subdomains $Q_i^\gamma\subset\R^2$ for the definition of $(\partial g^*)_\gamma$.}
        \label{fig:greg}
    \end{figure}

    This pointwise characterization allows obtaining expressions for the Moreau--Yosida approximation and the complementarity formulation of $u \in \partial \calG^*(p)$.

    \section{Optimality conditions and structure}\label{sec:optimality}

    We now discuss the properties of solutions $(\bar u,\bar p)$ to system~\eqref{eq:formal_opt}.
    Specifically, let 
    \begin{equation}
        U=L^2(D;\R^2) \qquad\text{ and }\qquad
        \calG:U\to\R,\quad
        \calG(u) = \int_D g(u(x))\,dx
    \end{equation}
    with $g$ given by~\eqref{eq:pointwise:switching}.
    The functional $\calF$ will be assumed to be a tracking term of the form
    \begin{equation}\label{eq:choice_f}
        \calF(u) = \frac12 \norm{Su - z}_{Y}^2
    \end{equation}
    for a Hilbert space $Y=Y^*$ (e.g., $Y=L^2([0,T]\times\Omega)$), given $z\in Y$, and a bounded linear control-to-observation mapping $S:U\to Y$. We further assume the existence of a Banach space $V\hookrightarrow L^r(D;\R^2)$ with $r>2$ such that the adjoint $S^*:Y\to U$ maps continuously into $V$. The optimality system~\eqref{eq:formal_opt} is then given by
    \begin{equation}\label{eq:opt_switching}
        \left\{\begin{aligned}
                \bar p &= -S^*(S\bar u-z),\\
                \bar u &\in \partial\calG^*(\bar p).
        \end{aligned}\right.
        \tag{OS}
    \end{equation}
    From \eqref{eq:gbiconj_bound} it follows that $\calG^{**}$ is radially unbounded. Hence, $\calF$ and $\calG$ satisfy assumption \eqref{eq:assumption_f}, and \cref{thm:existence} yields existence of a solution $(\bar u,\bar p)\in U\times U$ (which is unique if $S$ is injective).

    Using \cref{sec:switching:subdifferential} and the pointwise characterization of the subdifferential of integral functionals (see, e.g., \cite[Proposition 16.50]{Bauschke:2011}), the second relation in \eqref{eq:opt_switching} implies that for almost all $x\in D$,
    \begin{equation}\label{eq:opt_g}
        \begin{aligned}[t]
            \bar u(x) &\in  [\partial \calG^*(p)](x) = \partial g^*(p(x))\\
                      &=
            \begin{cases}
                \left(\left\{\tfrac1\alpha \bar p_1(x)\right\},\{0\}\right) & \text{if } \bar p(x) \in Q_1= \set{q}{|q_1| > |q_2| \text{ and } |q_2| < \sqrt{2\alpha\beta}},\\
                \left(\{0\},\left\{\tfrac1\alpha \bar p_2(x)\right\}\right) & \text{if } \bar p(x) \in Q_2= \set{q}{|q_2| > |q_1| \text{ and } |q_1| < \sqrt{2\alpha\beta}},\\
                \left(\left\{\tfrac1\alpha \bar p_1(x)\right\},\left\{\tfrac1\alpha \bar p_2(x)\right\}\right) & \text{if } \bar p(x) \in Q_0= \set{q}{|q_1|,|q_2| > \sqrt{2\alpha\beta}},\\
                \left(\left\{\tfrac1\alpha \bar p_1(x)\right\},\left[0,\tfrac1\alpha \bar p_2(x)\right]\right) & \text{if } \bar p(x) \in Q_{10}= \set{q}{|q_1| > |q_2|  \text{ and } |q_2| = \sqrt{2\alpha\beta}},\\
                \left(\left[0,\tfrac1\alpha \bar p_1(x)\right],\left\{\tfrac1\alpha \bar p_2(x)\right\}\right) & \text{if } \bar p(x) \in Q_{20}= \set{q}{|q_2| > |q_1|  \text{ and } |q_1| = \sqrt{2\alpha\beta}},\\
                \set{\left(\tfrac{t}{\alpha}\bar p_1(x),\tfrac{1-t}\alpha \bar p_2(x)\right)}{t\in[0,1]} & \text{if } \bar p(x) \in Q_{12}= \set{q}{|q_1| = |q_2|  \text{ and } |q_1| \leq \sqrt{2\alpha\beta}}.
            \end{cases}
        \end{aligned}
    \end{equation}
    We define the
    \emph{switching arc} (where at most one control is active, i.e., nonzero)
    \begin{align}
        \calA &= \set{x\in D}{\bar p(x) \in Q_1 \cup Q_2 \cup \{(0,0)\}},\\
        \intertext{the \emph{free arc} (where both controls are active)}
        \calI &= \set{x\in D}{\bar p(x) \in Q_0 \cup Q_{10} \cup Q_{20}},\\
        \intertext{and the \emph{singular arc}}
        \calS &= \set{x\in D}{\bar p(x) \in Q_{12}\setminus\{(0,0)\}}.
        \intertext{In a slight abuse of notation, we also introduce}
        \partial\calI &=  \set{x\in D}{\bar p(x)\in Q_{10}\cup Q_{20}}.
    \end{align}
    Clearly,
    \begin{equation}
        D = \calA \cup \calI \cup \calS.
    \end{equation}

    Let us address the question when the solution to system~\eqref{eq:opt_switching} will be optimal. For this purpose, we first estimate the duality gap \eqref{eq:dual_gap}.
    \begin{lemma}
        If $(\bar u,\bar p)\in U\times U$ satisfies $\bar u \in\partial\calG^*(\bar p)$,  then
        \begin{equation}
            \delta(\bar u,\bar p) \leq \beta |\partial\calI| + 2\beta |\calS|.
        \end{equation}
    \end{lemma}
    \begin{proof}
        We discriminate pointwise in the definition \eqref{eq:dual_gap} based on the value of $\bar p(x)$ for almost every $x\in D$.
        \begin{enumerate}[(i)]
            \item $\bar p(x) \in Q_1$. In this case, the relation~\eqref{eq:opt_g} yields $\bar u_1(x) = \frac1\alpha \bar p_1(x)$ and $\bar u_2(x) = 0$, and thus
                \begin{equation}
                    g(\bar u(x)) + g^*(\bar p(x)) - \bar p(x)\cdot \bar u(x) = \frac1{2\alpha}  \bar p_1(x)^2 +  \frac1{2\alpha}  \bar p_1(x)^2 -  \frac1\alpha  \bar p_1(x)^2 = 0.
                \end{equation}
            \item $\bar p(x) \in Q_2$. In this case, the relation~\eqref{eq:opt_g} yields $\bar u_1(x) = 0$ and $\bar u_2(x) =  \frac1\alpha \bar p_2(x)$, and thus
                \begin{equation}
                    g(\bar u(x)) + g^*(\bar p(x)) - \bar p(x)\cdot \bar u(x) = \frac1{2\alpha}  \bar p_2(x)^2 +  \frac1{2\alpha} \bar p_2(x)^2 -  \frac1\alpha \bar p_2(x)^2 = 0.
                \end{equation}
            \item $\bar p(x) \in Q_0$. In this case, the relation~\eqref{eq:opt_g} yields $\bar u_1(x) =  \frac1\alpha \bar p_1(x)$ and $\bar u_2(x) =  \frac1\alpha \bar p_2(x)$, and thus
                \begin{equation}
                    \begin{multlined}
                        g(\bar u(x)) + g^*(\bar p(x)) - \bar p(x)\cdot \bar u(x) = \frac1{2\alpha} (\bar p_1(x)^2+\bar p_2(x)^2) + \beta  + \frac1{2\alpha} (\bar p_1(x)^2 + \bar p_2(x)^2)\\ - \beta  -  \frac1\alpha ( \bar p_1(x)^2 + \bar p_2(x)^2) = 0.
                    \end{multlined}
                \end{equation}
            \item $\bar p(x) \in Q_{10}$. In this case, the relation~\eqref{eq:opt_g} yields $\bar u_1(x) =  \frac1\alpha \bar p_1(x)$ and $\bar u_2(x) \in [0, \frac1\alpha \bar p_2(x)]$. Assume first that $\bar p_2(x)$ is positive, and that  $0< \bar u_2(x) < \frac1\alpha \bar p_2(x)$ (otherwise argue as in case (i) or (iii)). Then,
                \begin{equation}
                    \begin{aligned}
                        g(\bar u(x)) + g^*(\bar p(x)) - \bar p(x)\cdot \bar u(x) &= \frac1{2\alpha} \bar p_1(x)^2 + \frac\alpha2 \bar u_2(x)^2 + \beta  + \frac1{2\alpha} \bar p_1(x)^2\\
                        \MoveEqLeft[-1]  -  \frac1\alpha \bar p_1(x)^2 -  \bar p_2(x)\bar u_2(x)\\
                        &= \frac\alpha2 u_2(x)^2 - \bar p_2(x)\bar u_2(x) + \beta.
                    \end{aligned}
                \end{equation}
                A simple calculus argument shows that the right-hand side is a monotonically decreasing function of $\bar u_2(x)$ on $(0,\frac1\alpha\bar p_2(x))$ and hence attains its supremum for $\bar u_2(x) = 0$, which implies that
                \begin{equation}
                    g(\bar u(x)) + g^*(\bar p(x)) - \bar p(x)\bar u(x)  < \beta
                \end{equation}
                for all $\bar u_2(x) \in (0,\frac1\alpha\bar p_2(x))$.
                For $\bar q_2(x)$ negative, we argue similarly.

            \item  $\bar p(x) \in Q_{20}$. In this case, the relation~\eqref{eq:opt_g} yields $\bar u_1(x) \in [0, \frac1\alpha \bar p_1(x)] $ and $\bar u_2(x)=  \frac1\alpha \bar p_2(x)$. Proceeding as in case (iv) yields
                \begin{equation}
                    g(\bar u(x)) + g^*(\bar p(x)) - \bar p(x)\bar u(x)  < \beta.
                \end{equation}

            \item  $\bar p(x) \in Q_{12}$. In this case, the relation~\eqref{eq:opt_g} yields $(\bar u_1(x),\bar u_2(x)) = \left(\frac{t}{\alpha}\bar p_1(x),\frac{1-t}{\alpha}\bar p_2(x)\right)$ for some $t\in[0,1]$. Furthermore, we have that $|\bar p_1(x)|=|\bar p_2(x)| \leq \sqrt{2\alpha\beta}$.         

                First, if $\bar p(x) = (0,0)\in Q_{12}$, this implies that $\bar u(x) = (0,0)$ and hence
                \begin{equation}
                    g(\bar u(x)) + g^*(\bar p(x)) - \bar p(x)\bar u(x) = 0.
                \end{equation}

                For $\bar p(x) \neq (0,0)$, we obtain
                \begin{equation}
                    \begin{aligned}
                        g(\bar u(x)) + g^*(\bar p(x)) - \bar p(x)\cdot \bar u(x) &= \frac\alpha2 \bar u_1(x)^2 + \frac\alpha2 \bar u_2(x)^2 + \beta  + \frac1{2\alpha} \bar p_1(x)^2\\
                        \MoveEqLeft[-1]  -  \bar p_1(x)\bar u_1(x) -  \bar p_2(x)\bar u_2(x)\\
                        &= \frac1{2\alpha}(t^2 - t +1)\bar p_1(x)^2 + \frac1{2\alpha}(t^2 - t)\bar p_2(x)^2 + \beta.
                    \end{aligned} 
                \end{equation}
                Both expressions in parentheses are convex quadratic functions of $t\in[0,1]$ and hence attain their supremum at $t=0$ and $t=1$. Together with $|\bar p_1(x)|\leq \sqrt{2\alpha\beta}$ this implies that
                \begin{equation}
                    g(\bar u(x)) + g^*(\bar p(x)) - \bar p(x)\bar u(x)  \leq 2\beta.
                \end{equation}
        \end{enumerate}
        Integrating over $D$ now yields the claim.
    \end{proof}

    From \cref{lem:dual_gap} we obtain the following characterization of (sub)optimality of solutions.
    \begin{theorem}
        If $(\bar u,\bar p)\in U\times U$ satisfies~\eqref{eq:opt_switching}, then for any $u\in U$,
        \begin{equation}
            \calJ(\bar u) \leq \calJ(u) + \beta(|\partial\calI|+2|\calS|).
        \end{equation}
        Hence if $\partial\calI$ and $\calS$ are sets of Lebesgue measure zero, $\bar u$ is a solution to~\eqref{eq:formal_prob}.
    \end{theorem}

    We next investigate the behavior of $\calI$ and $\calS$ as $\beta \to \infty$. For this purpose, we denote by $(u_\beta,p_\beta)$ the solution to~\eqref{eq:opt_switching} for given $\beta>0$, with corresponding free arc $\calI_\beta$. Note that the value of $\beta$ does not appear in the relation~\eqref{eq:subdiff_inclusion} except as part of the case distinction, and hence $\beta\to\infty$ does not necessarily imply that $u_\beta \to 0$.
    \begin{theorem}\label{thm:singular_set}
        Let $\alpha>0$ be fixed and let $(u_\beta,p_\beta)$ satisfy~\eqref{eq:opt_switching}. Then, $|\calI_\beta|\to 0$ as $\beta\to\infty$. 
    \end{theorem}
    \begin{proof}
        We use the minimizing properties of $u_\beta$ with respect to $\calF+\calG^{**}$ by making use of $g^{**}$ computed in \cref{sec:biconjugate}; see \eqref{eq:gbiconj}.
        Note that from the subdifferential inclusion~\eqref{eq:opt_g}, we can see that $u_\beta(x) \in D_0$ if and only if $p_\beta(x) \in \overline{Q_0}$.
        Since $g^{**}(0) = 0$, we have that 
        \begin{equation}
            \calG^{**}(u_\beta) \leq \calF(u_\beta) + \calG^{**}(u_\beta) \leq \calF(0) =:K,
        \end{equation}
        i.e., the family $\{\calG^{**}(u_\beta)\}_{\beta>0}$ is bounded. 
        We thus have for the free arc
        \begin{equation}
            \calI_\beta = \set{x\in D}{|p_{\beta,1}(x)|, |p_{\beta,2}(x)|\geq\sqrt{2\alpha\beta}} = \set{x\in D}{|u_{\beta,1}(x)|,|u_{\beta,2}(x)| \geq\sqrt{\tfrac{2\beta}{\alpha}}}
        \end{equation}
        that
        \begin{equation}\label{eq:freearc_bounded}
            K \geq \int_D g^{**}(u_\beta(x))\,dx \geq \int_{\calI_\beta} \frac\alpha2\left(|u_{\beta,1}(x)|^2+|u_{\beta,2}(x)|^2\right)+\beta\,dx  \geq \beta |I_\beta|,
        \end{equation}
        where the right-hand side remains bounded as $\beta \to \infty$ if and only if the second term goes to zero as claimed.
    \end{proof}
    Note that $\partial\calI_\beta\subset\calI_\beta$ and hence, from the estimate~\eqref{eq:freearc_bounded}, the corresponding optimality gap $\beta|\partial\calI_\beta|$ remains bounded for $\beta\to \infty$.

    If $p_\beta$ is uniformly bounded pointwise almost everywhere, we can deduce that $\calI_\beta$ must vanish for some sufficiently large (finite) value of $\beta$. 
    \begin{theorem}
        If $V\hookrightarrow L^\infty(D)$, then there exists a $\beta_0>0$ such that $|\calI_\beta|=0$ for all $\beta\geq \beta_0$.
    \end{theorem}
    \begin{proof}
        Due to the estimate~\eqref{eq:freearc_bounded} and the definition of $\calG^{**}$, the family $\{u_\beta\}_{\beta>0}$ is bounded in $U$. Hence $\{Su_\beta\}_{\beta>0}$ and thus $\{F'(Su_\beta)\}_{\beta>0}$ are bounded in $Y$ and $Y^*$, respectively. Since $S^*$ maps continuously to $L^\infty(D)$, this implies that $\{p_\beta\}_{\beta>0} = \{-S^*F'(Su_\beta)\}_{\beta>0}$ is uniformly bounded pointwise almost everywhere by a constant $M>0$. Choosing $\beta_0$ such that $M>\sqrt{2\alpha\beta_0}$, we obtain from the subdifferential inclusion~\eqref{eq:opt_g} that $Q_0=Q_{10}=Q_{20}=\emptyset$, which yields the claim.
    \end{proof}

    \begin{remark}
        The above theorem is a result in the spirit of exact penalization as in, e.g., \cite{Gugat:2009}. However, it does not yield an exact penalization of the switching condition $u_1u_2=0$ almost everywhere since the singular set $\calS$ cannot be controlled fully. It appears difficult to give a sufficient condition for $\calS$ to be empty, since on this set neither $\calF(u)$ nor $\calG(u)$ yield enough information to decide \emph{which} component of $u$ should be active. On the other hand, since $|\bar p_1(x)|=|\bar p_2(x)|$ has to hold on the singular arc, we can expect $|\calS|$ to be small. We shall comment on the cardinality of $\calS$ for the numerical examples. Direct extensions of the concepts in \cite{Gugat:2009} are not possible, since sparsity-promoting or exact penalty functionals of the type $|\cdot|^p$ with $p\in [0,1]$ on the controls do not lead to well-posed optimal control problems.
    \end{remark}

    \section{Numerical solution}
    \label{sec:switching:solution}

    We return to the Moreau--Yosida regularization of the optimality system~\eqref{eq:opt_switching}: For given $\gamma>0$, find $(u_\gamma,p_\gamma)\in U\times U$ satisfying
    \begin{equation}\label{eq:opt_switching_reg}
        \left\{\begin{aligned}
                p_\gamma &= -S^*(S u_\gamma-z),\\
                u_\gamma &= H_\gamma(p_\gamma).
        \end{aligned}\right.
    \tag{OS$_\gamma$}
\end{equation}
Since $\calF'(u) = S^*(S u-z)$ is linear and bounded, assumption \eqref{eq:assumption_f_bd} is clearly satisfied; in addition, the explicit characterization of $\partial\calG^*$ in \cref{sec:pointwise} immediately yields that
$\inf_{q\in\partial\calG^*(p)}\norm{q}_U \leq \frac1\alpha\norm{p}_U$, and hence assumption \eqref{eq:assumption_g_bd} holds. From \cref{thm:existence_reg} and \cref{thm:convergence}, we thus obtain existence of a solution (which is unique if $S$ is injective) and convergence to a solution of \eqref{eq:opt_switching} as $\gamma \to 0$.
For later reference, we note that the mapping properties of $S^*$ imply that $p_\gamma \in V$.

\bigskip

The solution to~\eqref{eq:opt_switching_reg} can be computed using a semismooth Newton method. We first show that $H_\gamma$ is Newton-differentiable. Recall that $H_\gamma$ is defined pointwise almost everywhere by
\begin{equation}
    [H_\gamma(p)](x) = h_\gamma(p(x)) := (\partial g^*)_\gamma(p(x)),
\end{equation}
and that $h_\gamma$ is globally Lipschitz continuous with constant $\gamma^{-1}$ by \cref{thm:convex:my}\,(iii). Hence, $h_\gamma$ is directionally differentiable almost everywhere. In addition, $h_\gamma$ is piecewise differentiable, and hence its directional derivative 
\begin{equation}
    h_\gamma'(q;\delta q) := \lim_{t\to 0} \frac1t(h_\gamma(q+t\delta q)-h_\gamma(q))
\end{equation}
at $q$ in direction $\delta q$ satisfies
\begin{equation}
\lim_{|\delta q|\to 0} \frac1{|\delta q|} |h_\gamma'(q+\delta q;\delta q)-h_\gamma'(q;\delta q)| = 0\qquad\text{for almost all $q$.}
\end{equation}
Together we obtain that $h_\gamma$ is semismooth; see, e.g., \cite[Theorem 8.2]{Kunisch:2008a} or \cite[Proposition 2.7]{Ulbrich:2011}; see also \cite[Proposition 2.26]{Ulbrich:2011}.

This implies that the superposition operator $H_\gamma$ is Newton-differentiable from $V\hookrightarrow L^r(D;\R^2)$ to $L^2(D;\R^2)$ for any $r>2$; see, e.g., \cite[Example 8.12]{Kunisch:2008a} or \cite[Theorem 3.49]{Ulbrich:2011}. Its Newton derivative  will be denoted by $D_NH_\gamma:V\to U$, and it is given pointwise almost everywhere at $p$ in direction $\delta p$ by a measurable selection
\begin{equation}\label{eq:newton_superposition}
    [D_NH_\gamma(p)\delta p](x) \in \partial_C h_\gamma(p(x))\delta p(x),
\end{equation}
where $\partial_C h_\gamma(q)$ is the Clarke derivative, which for piecewise differentiable functions is given by the convex hull of the piecewise derivatives at each point. Specifically, for $h_\gamma$ given in \cref{sec:switching:resolvent}, a Newton derivative $D_N h_\gamma(q) \in \partial_C h_\gamma(q)$ is given by
\begin{equation}
    D_N h_\gamma(q) = 
    \begin{cases}
        \mathrm{diag}\left(\tfrac1{\alpha+\gamma},0\right)
        & \text{if }q\in Q_1^\gamma,\\
        \mathrm{diag}\left(0,\tfrac1{\alpha+\gamma}\right)
        & \text{if }q\in Q_2^\gamma,\\
        \mathrm{diag}\left(\tfrac1{\alpha+\gamma},\tfrac1{\alpha+\gamma}\right)
        & \text{if } q\in Q_0^\gamma,\\
        \mathrm{diag}\left(\tfrac1{\alpha+\gamma},\tfrac1\gamma\right)
        & \text{if } q\in Q_{10}^\gamma,\\
        \mathrm{diag}\left(\tfrac1\gamma,\tfrac1{\alpha+\gamma}\right)
        & \text{if } q\in Q_{20}^\gamma,\\
        \mathrm{diag}\left(\tfrac1{\gamma},\tfrac1{\gamma}\right)
        & \text{if } q\in Q_{00}^\gamma,\\
        \frac1{\gamma(2\alpha+\gamma)}
        \begin{pmatrix}
            (\alpha+\gamma) & \sign(q_1q_2){\alpha} \\ 
            \sign(q_1q_2){\alpha} & (\alpha+\gamma)  
        \end{pmatrix}
        & \text{if } q\in Q_{12}^\gamma,
    \end{cases}
\end{equation}
where $\mathrm{diag}(\cdot,\cdot)$ denotes the $2\times 2$ diagonal matrix with the given entries.

In the sequel, we shall require the following two properties of the Newton derivative.
\begin{lemma}\label{lem:newton_bounded}
    For all $p\in V$ and $\delta p \in V$, we have
    \begin{align}
        \scalprod{D_NH_\gamma(p)\delta p,\delta p}_{U} &\geq 0,\\
        \norm{D_NH_\gamma (p)\delta p}_{U} &\leq \frac1\gamma \norm{\delta p}_U.
    \end{align}
\end{lemma}
\begin{proof}
    Recall from \cref{thm:convex:my} that $h_\gamma$ is the derivative of the convex functional $(g^*)_\gamma$ and hence is monotone. 
    Therefore we have for all $t>0$, almost all $q$, and all $\delta q$ that
    \begin{equation}
        0 \leq (h_\gamma(q+t\delta q)-h_\gamma(q))\cdot(q+t\delta q -q)
        = \frac1t \left(h(q+t\delta q) - h_\gamma(q)\right)\cdot (t^2\delta q).
    \end{equation}
    Dividing by $t^2>0$ and taking the limit as $t\to 0$ yields
    \begin{equation}\label{eq:newton_bounded1}
        h_\gamma'(q;\delta q)\cdot\delta q \geq 0.
    \end{equation}
    Similarly, since $h_\gamma$ is globally Lipschitz with constant $\gamma^{-1}$, we have  for all $t>0$, almost all $q$, and all~$\delta q$ that
    \begin{equation}
        \frac1t |h_\gamma(q+t\delta q)-h_\gamma(q)| \leq \frac1\gamma |\delta q|.
    \end{equation}
    Taking again the limit as $t\to 0$ yields
    \begin{equation}\label{eq:newton_bounded2}
        |h_\gamma' (q;\delta q)| \leq \frac1\gamma |\delta q|.
    \end{equation}   
    As a consequence, all elements in the Clarke derivative satisfy the inequalities~\eqref{eq:newton_bounded1} and~\eqref{eq:newton_bounded2}. Since $D_NH_\gamma(p)$ is taken as a measurable selection from $\partial_C h_\gamma(p(\cdot))$, the claim follows by substitution and integration over $D$.
\end{proof}

\bigskip

To apply a semismooth Newton method to \eqref{eq:opt_switching_reg}, we first introduce the state $y_\gamma:=S(u_\gamma)\in Y$ and eliminate $u_\gamma$, thus obtaining the equivalent optimality system
\begin{equation}\label{eq:optimality_red}
    \left\{\begin{aligned}
            y_\gamma &= S H_\gamma(p_\gamma),\\
            p_\gamma &= -S^* (y_\gamma-z).
    \end{aligned}\right.
\end{equation}
Considering the system~\eqref{eq:optimality_red} as an operator equation from $Y\times V$ to $Y\times V$, a semismooth Newton step for its solution consists in computing $(\delta y,\delta p)\in Y\times V$ for given $(y^k,p^k)\in Y\times V$ such that
\begin{equation}\label{eq:newton_step}
    \left\{\begin{aligned}
            \delta y - S D_NH_\gamma(p^k)\delta p &= -y^k+ S H_\gamma(p^k) ,\\
            \delta p + S^*\delta y &= -p^k - S^* (y^k-z),
    \end{aligned}\right.
\end{equation}
and setting $y^{k+1} = y^k +\delta y$ and $p^{k+1} = p^k + \delta p$. 

To show superlinear convergence of this iteration, it remains to show uniform solvability of each Newton step.
\begin{proposition}\label{thm:newton_bound}
    For any $(y,p)\in Y\times V$ and $(w_1,w_2)\in Y\times V$, the system
    \begin{equation}\label{eq:newton_solve}
        \left\{\begin{aligned}
                \delta y + S D_NH_\gamma(p)\delta p &= w_1 ,\\
                \delta p - S^*\delta y &= w_2,
        \end{aligned}\right.
    \end{equation}
    has a solution $(\delta y,\delta p)\in Y\times V$ which satisfies
    \begin{equation}
        \norm{\delta y}_{Y} + \norm{\delta p}_V \leq C(\norm{w_1}_{Y} + \norm{w_2}_{V}).
    \end{equation}
\end{proposition}
\begin{proof}
    Eliminating $\delta p = S^*\delta y + w_2\in V$, we obtain that~\eqref{eq:newton_solve} is equivalent to
    \begin{equation}\label{eq:newton_red}
        \delta y + SD_NH_\gamma(p)S^*\delta y = w_1 + SD_NH_\gamma(p) w_2.
    \end{equation}
    Since $S^*$ is linear and bounded from $Y$ to $V$ and $D_NH_\gamma$ is monotone on $V$ from \cref{lem:newton_bounded}, the operator $SD_NH_\gamma(p)S^*$ is maximally monotone from $Y$ to $Y$; see, e.g., \cite[Propositions 20.10, 20.24]{Bauschke:2011}. Minty's theorem thus yields existence of a solution $\delta y\in Y$ and hence of a corresponding $\delta p \in V$; see, e.g.,  \cite[Proposition 21.1]{Bauschke:2011}.

    Taking the inner product of equation~\eqref{eq:newton_red} with $\delta y$ and using \cref{lem:newton_bounded} with $S^*\delta y\in V\hookrightarrow U$ implies that
    \begin{equation}
        \begin{aligned}
            \norm{\delta y}^2_{Y} &\leq  \scalprod{\delta y,\delta y}_{Y} + \scalprod{D_NH_\gamma(p)(S^*\delta y),S^*\delta y}_{U} \\
                                  &= \scalprod{w_1,\delta y}_{Y} + \scalprod{D_NH_\gamma(p)w_2,S^*\delta y}_{U} \\
                                  &\leq \norm{w_1}_{Y}\norm{\delta y}_{Y} + \norm{D_NH_\gamma(p)w_2}_{U}\norm{S^*\delta y}_U\\
                                  &\leq \left(\norm{w_1}_{Y} + \frac{C}\gamma\norm{w_2}_{V}\right)\norm{\delta y}_{Y},
        \end{aligned}
    \end{equation}
    using the boundedness of $S^*$ from $Y$ to $V$ and \cref{lem:newton_bounded} with $w_2\in V\hookrightarrow U$.
    The second equation of~\eqref{eq:newton_solve} then yields
    \begin{equation}
        \begin{split} 
            \norm{\delta p}_V \leq C\norm{w_1}_{Y} + \left(1+\frac{C^2}\gamma\right)\norm{w_2}_{V}.
            \qedhere
        \end{split}
    \end{equation}
\end{proof}
As a consequence of the Newton differentiability of $H_\gamma$ and of \cref{thm:newton_bound}, we obtain the following result; see, e.g., \cite[Theorem 8.16]{Kunisch:2008a}, \cite[Chapter 3.2]{Ulbrich:2011}.
\begin{theorem}
    The semismooth Newton iteration~\eqref{eq:newton_step} converges locally superlinearly in $Y\times V$.
\end{theorem}  
Since the right-hand side of the Newton system \eqref{eq:newton_step} is linear apart from the term $H_\gamma(p^k)$, we can use the following termination criterion for the Newton iteration: If all active sets $A_i(p) = \set{x\in\Omega}{p(x)\in Q^\gamma_i}$ coincide for $p^k$ and $p^{k+1}$, and the control is computed as $u^{k+1}=H_\gamma(p^{k+1})$, then $(u^{k+1}, p^{k+1})$ satisfies \eqref{eq:opt_switching_reg}; see, e.g., \cite[Remark 7.1.1]{Kunisch:2008a}.

This can be used as part of a continuation strategy to deal with the local convergence behavior of Newton methods: Starting with  $\gamma^0$ large and $(y^0,p^0)=(0,0)$, we solve the regularized optimality system \eqref{eq:opt_switching_reg} using the semismooth Newton iteration \eqref{eq:newton_step}. If the iteration converges for some $\gamma^m$ (in the sense that all active sets coincide), we reduce $\gamma^{m+1} = \frac1{10}\gamma^{m}$  and solve the system \eqref{eq:opt_switching_reg} again with the solution for $\gamma^{m}$ as the starting point. This procedure is terminated if the Newton iteration converges in a single step (assuming that the corresponding iterate then satisfies the system for smaller values of $\gamma$ as well) or if the Newton iteration fails to converge within a given number of steps (assuming that the system has then become too ill-conditioned for a stable numerical solution). In any case, the continuation is stopped when $\gamma^m \leq 10^{-16}$ is reached. 

While this strategy has proved robust for problems with scalar $L^1$- and $L^0$-type penalties, see e.g. \cite{IK:2012,CK:2013}, the situation is more delicate for the vector functional considered here; this is in particular the case when the singular arc $\calS$ is non-negligible and $D_N H_\gamma$ is not a diagonal matrix, where the continuation strategy failed in some cases to provide a good initial guess for the next Newton iteration.
We thus combine the semismooth Newton method with a backtracking line search along the Newton direction. In principle, this requires computation of $(\calG^*_\gamma)^*$ (or $\calF^*$ and $\calG^*_\gamma$); however, if the tracking term $\calF$ is strictly convex (as will be the case in the examples considered below), the system \eqref{eq:opt_switching_reg} is a sufficient as well as necessary condition and hence we can equivalently backtrack according to the residual norm of \eqref{eq:opt_switching_reg}. This was sufficient to achieve a robust and superlinear convergence in all examples.

\section{Numerical examples}\label{sec:examples}

We illustrate the behavior of the proposed approach and the structure of the resulting controls with two numerical examples. First, we consider an elliptic problem where the two control components each act along a strip in one coordinate direction. Specifically, we set $\Omega = [0,1]^2$, $D=[0,1]$, 
\begin{equation}
    \omega_1 = \set{(x_1,x_2)\in\Omega}{x_2 < \tfrac14},\qquad 
    \omega_2 = \set{(x_1,x_2)\in\Omega}{x_2 > \tfrac34},
\end{equation}
and consider the control-to-state mapping $S: u\mapsto y\in Y = L^2(\Omega)$ satisfying
\begin{equation}
    -\Delta y = Bu = \chi_{\omega_1}(x_1,x_2)u_1(x_1) + \chi_{\omega_2}(x_1,x_2)u_2(x_1).
\end{equation}
The target is
\begin{equation}
    z(x) = x_1\sin(2\pi x_1)\sin(2\pi x_2),
\end{equation}
see \cref{fig:ell_target}.
\begin{figure}[t]
    \centering
    \begin{tikzpicture}
        \node[anchor=south west,inner sep=0] at (0,0) {\includegraphics[width=0.5\textwidth]{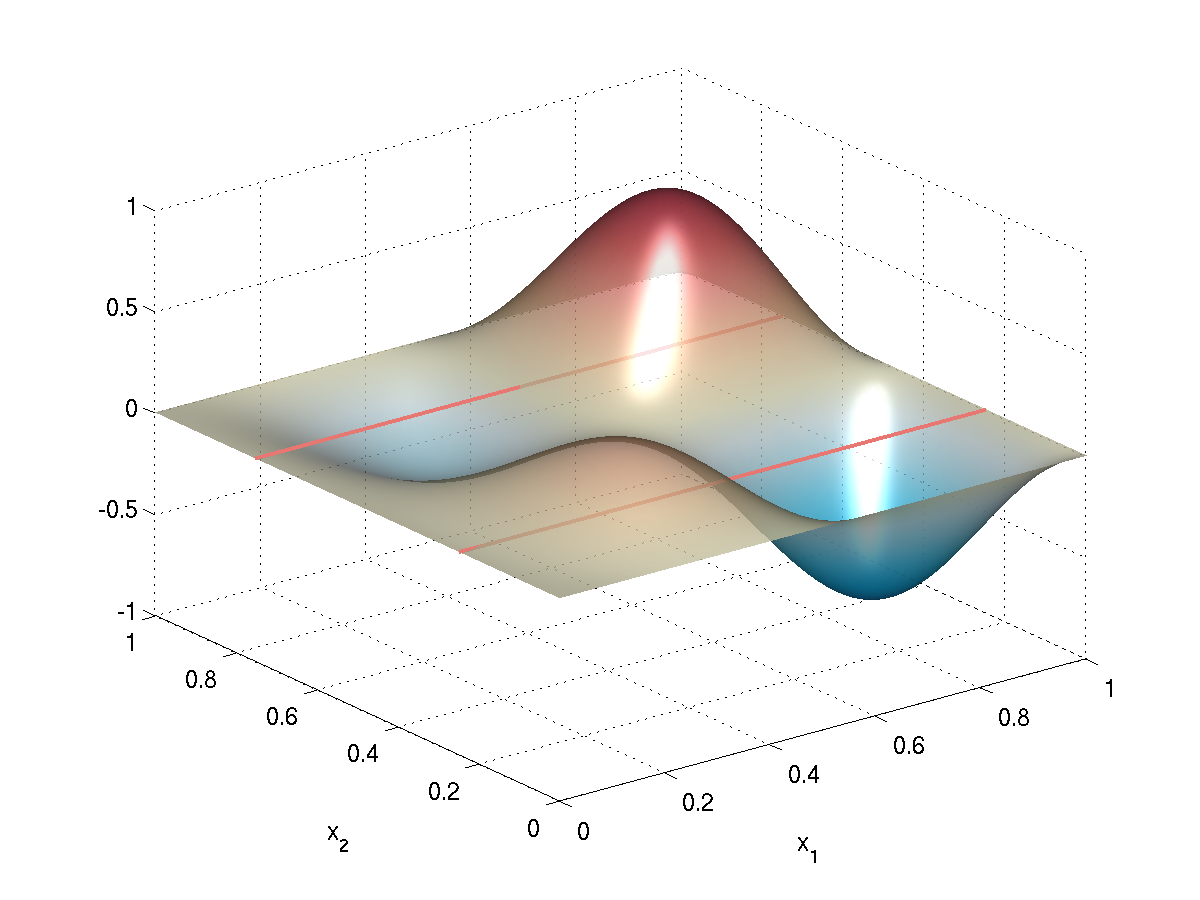}};
        \node at (3.5,2.1) {\footnotesize $\omega_1$};
        \node at (1.6,2.95) {\footnotesize $\omega_2$};
    \end{tikzpicture}
    \caption{Elliptic problem, target $z$ and control domains $\omega_1$, $\omega_2$}
    \label{fig:ell_target}
\end{figure}

The state $y$ and adjoint $p$ are discretized using piecewise linear finite elements based on a uniform triangulation $\mathcal{T}_h$ of the domain $\Omega$ with $N_h = 128\times128$ nodes. 
Since the control is eliminated, this can be interpreted as a variational discretization. 
Integration over the piecewise defined functions $H_\gamma(p_h)$ and $D_N H_\gamma(p_h)\delta p_h$ in the weak formulation of \eqref{eq:newton_step} is approximated by applying the mass matrix to the vector of nodal values; see \cite{CK:2013}. 
The control operator $B$ is approximated by forming the tensor product of the discrete indicator function of $\omega_i$ with the nodal values of $u_i$; the adjoint operator $B^*$ is approximated by the transpose of this matrix in order to preserve symmetry. 
The ``globalized'' semismooth Newton method with continuation and line searches described above is applied to the discretized system. 
The continuation is started at $\gamma^0=1$ and the backtracking is performed in steps of $\tau_i = 2^{-i}$ for $i=0,\dots,40$; if $\tau_i < 10^{-12}$,  the Newton iteration is restarted with reduced $\gamma$. 
Since we no longer perform full Newton steps, we augment the termination criterion for the Newton iteration with an additional check for the residual norm in the optimality system, i.e., we terminate if all active sets coincide and the residual is smaller than $10^{-6}$. 
A Matlab implementation of the described algorithm can be downloaded from \url{https://github.com/clason/switchingcontrol}.

\pgfplotsset{cycle list/Dark2-3}
\begin{figure}[t]
    \centering
    \begin{subfigure}[t]{0.475\textwidth}
        \begin{tikzpicture}

\begin{axis}[%
width=\textwidth,
xmin=0,
xmax=1,
ymin=-0.015,
ymax=0.015,
xlabel={$x_1$},
axis y line=left,
axis x line=middle,
legend style={legend cell align=left,align=left,draw=none,font=\footnotesize},
legend entries={$u_1$,$u_2$},
legend pos=north west
]
\addplot +[
line width=1,
]
table[row sep=crcr]{
0 2.10134069883681e-05\\
0.0078740157480315 0.000124071627733977\\
0.015748031496063 0.00024737157088374\\
0.0236220472440945 0.000371136545282718\\
0.031496062992126 0.000495520430096732\\
0.0393700787401575 0.000620620310039802\\
0.047244094488189 0.000746477730280378\\
0.0551181102362205 0.000873080098802991\\
0.062992125984252 0.00100036223487361\\
0.0708661417322835 0.00112820806158644\\
0.078740157480315 0.00125645243981356\\
0.0866141732283465 0.00138488314023827\\
0.094488188976378 0.00151324294952952\\
0.102362204724409 0.00164123190611081\\
0.110236220472441 0.00176850966039294\\
0.118110236220472 0.00189469795377856\\
0.125984251968504 0.00201938321020873\\
0.133858267716535 0.00214211923350858\\
0.141732283464567 0.00226243000330307\\
0.149606299212598 0.00237981256181519\\
0.15748031496063 0.00249373998342927\\
0.165354330708661 0.00260366441850285\\
0.173228346456693 0.00270902020254208\\
0.181102362204724 0.00280922702152005\\
0.188976377952756 0.00290369312381427\\
0.196850393700787 0.00299181856897102\\
0.204724409448819 0.0030729985032696\\
0.21259842519685 0.00314662645186138\\
0.220472440944882 0.00321209761709453\\
0.228346456692913 0.00326881217250934\\
0.236220472440945 0.0033161785418982\\
0.244094488188976 0.00335361665277085\\
0.251968503937008 0.00338056115354913\\
0.259842519685039 0.00339646458383508\\
0.267716535433071 0.00340080048715357\\
0.275590551181102 0.00339306645566331\\
0.283464566929134 0.0033727870964598\\
0.291338582677165 0.00333951690925749\\
0.299212598425197 0.00329284306543794\\
0.307086614173228 0.00323238807868117\\
0.31496062992126 0.00315781235770592\\
0.322834645669291 0.00306881663240712\\
0.330708661417323 0\\
0.338582677165354 0\\
0.346456692913386 0\\
0.354330708661417 0\\
0.362204724409449 0\\
0.37007874015748 0\\
0.377952755905512 0\\
0.385826771653543 0\\
0.393700787401575 0\\
0.401574803149606 0\\
0.409448818897638 0\\
0.417322834645669 0\\
0.425196850393701 0\\
0.433070866141732 0\\
0.440944881889764 -5.426871817271e-05\\
0.448818897637795 -0.000374772448488796\\
0.456692913385827 -0.000707034647027749\\
0.464566929133858 -0.00105040265219784\\
0.47244094488189 -0.00140417056268026\\
0.480314960629921 -0.00176758057196867\\
0.488188976377953 -0.00213982450765599\\
0.496062992125984 -0.00252004557234538\\
0.503937007874016 -0.00290734028240272\\
0.511811023622047 -0.00330076060012402\\
0.519685039370079 -0.00369931625425599\\
0.52755905511811 -0.00410197724318477\\
0.535433070866142 -0.00450767651449781\\
0.543307086614173 -0.0049153128140301\\
0.551181102362205 -0.00532375369692952\\
0.559055118110236 -0.00573183869271936\\
0.566929133858268 -0.00613838261580066\\
0.574803149606299 -0.00654217901232515\\
0.582677165354331 -0.0069420037338822\\
0.590551181102362 -0.00733661862798274\\
0.598425196850394 -0.00772477533489082\\
0.606299212598425 -0.00810521917995042\\
0.614173228346457 -0.00847669315018376\\
0.622047244094488 -0.00883794194359771\\
0.62992125984252 -0.00918771607932954\\
0.637795275590551 -0.00952477605649169\\
0.645669291338583 -0.00984789654934036\\
0.653543307086614 -0.0101558706261941\\
0.661417322834646 -0.0104475139793664\\
0.669291338582677 -0.010721669153255\\
0.677165354330709 -0.0109772097576429\\
0.68503937007874 -0.0112130446532229\\
0.692913385826772 -0.0114281220963491\\
0.700787401574803 -0.0116214338300491\\
0.708661417322835 -0.0117920191083997\\
0.716535433070866 -0.0119389686414727\\
0.724409448818898 -0.0120614284481872\\
0.732283464566929 -0.0121586036048983\\
0.740157480314961 0\\
0.748031496062992 0\\
0.755905511811024 0\\
0.763779527559055 0\\
0.771653543307087 0\\
0.779527559055118 0\\
0.78740157480315 0\\
0.795275590551181 0\\
0.803149606299213 0\\
0.811023622047244 0\\
0.818897637795276 0\\
0.826771653543307 0\\
0.834645669291339 0\\
0.84251968503937 0\\
0.850393700787402 0\\
0.858267716535433 0\\
0.866141732283465 0\\
0.874015748031496 0\\
0.881889763779528 0\\
0.889763779527559 0\\
0.897637795275591 0\\
0.905511811023622 0\\
0.913385826771654 0\\
0.921259842519685 0\\
0.929133858267717 0\\
0.937007874015748 0\\
0.94488188976378 0\\
0.952755905511811 0\\
0.960629921259842 0\\
0.968503937007874 0\\
0.976377952755906 0\\
0.984251968503937 0\\
0.992125984251969 0\\
1 0\\
};
\addplot +[
line width=1,
]
table[row sep=crcr]{
0 0\\
0.0078740157480315 0\\
0.015748031496063 0\\
0.0236220472440945 0\\
0.031496062992126 0\\
0.0393700787401575 0\\
0.047244094488189 0\\
0.0551181102362205 0\\
0.062992125984252 0\\
0.0708661417322835 0\\
0.078740157480315 0\\
0.0866141732283465 0\\
0.094488188976378 0\\
0.102362204724409 0\\
0.110236220472441 0\\
0.118110236220472 0\\
0.125984251968504 0\\
0.133858267716535 0\\
0.141732283464567 0\\
0.149606299212598 0\\
0.15748031496063 0\\
0.165354330708661 0\\
0.173228346456693 0\\
0.181102362204724 0\\
0.188976377952756 0\\
0.196850393700787 0\\
0.204724409448819 0\\
0.21259842519685 0\\
0.220472440944882 0\\
0.228346456692913 0\\
0.236220472440945 0\\
0.244094488188976 0\\
0.251968503937008 0\\
0.259842519685039 0\\
0.267716535433071 0\\
0.275590551181102 0\\
0.283464566929134 0\\
0.291338582677165 0\\
0.299212598425197 0\\
0.307086614173228 0\\
0.31496062992126 0\\
0.322834645669291 0\\
0.330708661417323 -0.00296541383043593\\
0.338582677165354 -0.00284712939978466\\
0.346456692913386 -0.00271379614061558\\
0.354330708661417 -0.0025652965601745\\
0.362204724409449 -0.00240156337914539\\
0.37007874015748 -0.00222258102841023\\
0.377952755905512 -0.00202838696558152\\
0.385826771653543 -0.00181907280557865\\
0.393700787401575 -0.00159478526011221\\
0.401574803149606 -0.0013557268815039\\
0.409448818897638 -0.001102156606854\\
0.417322834645669 -0.000834390099171244\\
0.425196850393701 -0.000552799882697651\\
0.433070866141732 -0.000257815270292344\\
0.440944881889764 0\\
0.448818897637795 0\\
0.456692913385827 0\\
0.464566929133858 0\\
0.47244094488189 0\\
0.480314960629921 0\\
0.488188976377953 0\\
0.496062992125984 0\\
0.503937007874016 0\\
0.511811023622047 0\\
0.519685039370079 0\\
0.52755905511811 0\\
0.535433070866142 0\\
0.543307086614173 0\\
0.551181102362205 0\\
0.559055118110236 0\\
0.566929133858268 0\\
0.574803149606299 0\\
0.582677165354331 0\\
0.590551181102362 0\\
0.598425196850394 0\\
0.606299212598425 0\\
0.614173228346457 0\\
0.622047244094488 0\\
0.62992125984252 0\\
0.637795275590551 0\\
0.645669291338583 0\\
0.653543307086614 0\\
0.661417322834646 0\\
0.669291338582677 0\\
0.677165354330709 0\\
0.68503937007874 0\\
0.692913385826772 0\\
0.700787401574803 0\\
0.708661417322835 0\\
0.716535433070866 0\\
0.724409448818898 0\\
0.732283464566929 0\\
0.740157480314961 0.0122297631431904\\
0.748031496062992 0.0122747394474308\\
0.755905511811024 0.012292448872702\\
0.763779527559055 0.0122823665692765\\
0.771653543307087 0.0122440434396422\\
0.779527559055118 0.0121771090259323\\
0.78740157480315 0.0120812742160163\\
0.795275590551181 0.0119563337580543\\
0.803149606299213 0.0118021685746761\\
0.811023622047244 0.0116187478685441\\
0.818897637795276 0.0114061310116351\\
0.826771653543307 0.011164469211193\\
0.834645669291339 0.0108940069459507\\
0.84251968503937 0.0105950831668946\\
0.850393700787402 0.0102681322575457\\
0.858267716535433 0.00991368474945476\\
0.866141732283465 0.00953236778935568\\
0.874015748031496 0.00912490535517884\\
0.881889763779528 0.00869211821890966\\
0.889763779527559 0.00823492365506382\\
0.897637795275591 0.0077543348943527\\
0.905511811023622 0.00725146032291539\\
0.913385826771654 0.00672750242829963\\
0.921259842519685 0.00618375649417069\\
0.929133858267717 0.00562160904650853\\
0.937007874015748 0.00504253605480249\\
0.94488188976378 0.00444810089244356\\
0.952755905511811 0.00383995206110122\\
0.960629921259842 0.00321982068426258\\
0.968503937007874 0.00258951777511112\\
0.976377952755906 0.00195093128307025\\
0.984251968503937 0.00130602292044003\\
0.992125984251969 0.00065682476216964\\
1 0.000111346258599944\\
};
\end{axis}
\end{tikzpicture}%
        \caption{$\alpha = 10^{-3}$, $\beta = 10^{-3}$\label{fig:ell_33}}
    \end{subfigure}
    \hfill
    \begin{subfigure}[t]{0.475\textwidth}
        \begin{tikzpicture}

\begin{axis}[%
width=\textwidth,
xmin=0,
xmax=1,
ymin=-0.015,
ymax=0.015,
xlabel={$x_1$},
axis y line=left,
axis x line=middle,
legend style={legend cell align=left,align=left,draw=none,font=\footnotesize},
legend entries={$u_1$,$u_2$},
legend pos=north west
]
\addplot +[
line width=1,
]
table[row sep=crcr]{
0 2.10131693495762e-05\\
0.0078740157480315 0.000124070251340663\\
0.015748031496063 0.000247368848122468\\
0.0236220472440945 0.000371132489094609\\
0.031496062992126 0.000495515060651343\\
0.0393700787401575 0.000620613655299498\\
0.047244094488189 0.000746469826551056\\
0.0551181102362205 0.000873070991269187\\
0.062992125984252 0.00100035197811599\\
0.0708661417322835 0.00112819672007955\\
0.078740157480315 0.00125644008840162\\
0.0866141732283465 0.00138486986458671\\
0.094488188976378 0.00151322884654989\\
0.102362204724409 0.00164121708435647\\
0.110236220472441 0.00176849424042292\\
0.118110236220472 0.00189468206848689\\
0.125984251968504 0.00201936700511618\\
0.133858267716535 0.0021421028670139\\
0.141732283464567 0.00226241364689041\\
0.149606299212598 0.00237979640021423\\
0.15748031496063 0.00249372421472444\\
0.165354330708661 0.00260364925418766\\
0.173228346456693 0.00270900586751438\\
0.181102362204724 0.00280921375401354\\
0.188976377952756 0.00290368117526141\\
0.196850393700787 0.0029918082037917\\
0.204724409448819 0.00307298999857973\\
0.21259842519685 0.00314662009709448\\
0.220472440944882 0.00321209371352909\\
0.228346456692913 0.00326881103269322\\
0.236220472440945 0.00331618048896044\\
0.244094488188976 0.00335362201960983\\
0.251968503937008 0.0033805702818836\\
0.259842519685039 0.00339647782310242\\
0.267716535433071 0.00340081819323577\\
0.275590551181102 0.00339308898941603\\
0.283464566929134 0.0033728148220117\\
0.291338582677165 0.00333955019205345\\
0.299212598425197 0.0032928822702156\\
0.307086614173228 0\\
0.31496062992126 0\\
0.322834645669291 0\\
0.330708661417323 0\\
0.338582677165354 0\\
0.346456692913386 0\\
0.354330708661417 0\\
0.362204724409449 0\\
0.37007874015748 0\\
0.377952755905512 0\\
0.385826771653543 0\\
0.393700787401575 0\\
0.401574803149606 0\\
0.409448818897638 0\\
0.417322834645669 0\\
0.425196850393701 0\\
0.433070866141732 0\\
0.440944881889764 -5.40428452085158e-05\\
0.448818897637795 -0.000374530575163943\\
0.456692913385827 -0.000706776108618577\\
0.464566929133858 -0.00105012678207148\\
0.47244094488189 -0.00140387669402191\\
0.480314960629921 -0.00176726803955366\\
0.488188976377953 -0.00213949264967735\\
0.496062992125984 -0.00251969373230773\\
0.503937007874016 -0.00290696781109456\\
0.511811023622047 -0.00330036685768099\\
0.519685039370079 -0.00369890061232888\\
0.52755905511811 -0.00410153908722744\\
0.535433070866142 -0.00450721524619158\\
0.543307086614173 -0.00491482785386302\\
0.551181102362205 -0.00532324448695087\\
0.559055118110236 -0.0057313046994921\\
0.566929133858268 -0.00613782333357721\\
0.574803149606299 -0.0065415939664748\\
0.582677165354331 -0.00694139248460188\\
0.590551181102362 -0.0073359807743268\\
0.598425196850394 -0.00772411051915975\\
0.606299212598425 -0.00810452709248339\\
0.614173228346457 -0.00847597353460555\\
0.622047244094488 -0.00883719460257697\\
0.62992125984252 -0.00918694088091241\\
0.637795275590551 -0.00952397294108338\\
0.645669291338583 -0.00984706553741669\\
0.653543307086614 -0.0101550118268359\\
0.661417322834646 -0.0104466275997228\\
0.669291338582677 -0.0107207555090563\\
0.677165354330709 -0.0109762692849014\\
0.68503937007874 -0.0112120779212832\\
0.692913385826772 -0.0114271298224728\\
0.700787401574803 -0.011620416895756\\
0.708661417322835 -0.0117909785778285\\
0.716535433070866 -0.0119379057820861\\
0.724409448818898 -0.0120603447542334\\
0.732283464566929 -0.0121575008238367\\
0.740157480314961 -0.012228642039688\\
0.748031496062992 -0.0122731026771244\\
0.755905511811024 -0.0122902866057681\\
0.763779527559055 -0.0122796705065091\\
0.771653543307087 -0.0122408069269492\\
0.779527559055118 -0.012173327164957\\
0.78740157480315 -0.0120769439704534\\
0.795275590551181 -0.0119514540560503\\
0.803149606299213 -0.0117967404077039\\
0.811023622047244 -0.0116127743871118\\
0.818897637795276 -0.011399617618188\\
0.826771653543307 -0.0111574236505793\\
0.834645669291339 -0.0108864393938481\\
0.84251968503937 -0.0105870063166323\\
0.850393700787402 -0.0102595614058094\\
0.858267716535433 -0.00990463788142215\\
0.866141732283465 -0.00952286566389028\\
0.874015748031496 -0.00911497159080518\\
0.881889763779528 -0.00868177938140844\\
0.889763779527559 -0.00822420934767071\\
0.897637795275591 -0.00774327785172208\\
0.905511811023622 -0.00724009651023709\\
0.913385826771654 -0.00671587114724571\\
0.921259842519685 -0.0061719004977295\\
0.929133858267717 -0.00560957466527197\\
0.937007874015748 -0.00503037333811746\\
0.94488188976378 0\\
0.952755905511811 0\\
0.960629921259842 0\\
0.968503937007874 0\\
0.976377952755906 0\\
0.984251968503937 0\\
0.992125984251969 0\\
1 0\\
};
\addplot +[
line width=1,
]
table[row sep=crcr]{
0 0\\
0.0078740157480315 0\\
0.015748031496063 0\\
0.0236220472440945 0\\
0.031496062992126 0\\
0.0393700787401575 0\\
0.047244094488189 0\\
0.0551181102362205 0\\
0.062992125984252 0\\
0.0708661417322835 0\\
0.078740157480315 0\\
0.0866141732283465 0\\
0.094488188976378 0\\
0.102362204724409 0\\
0.110236220472441 0\\
0.118110236220472 0\\
0.125984251968504 0\\
0.133858267716535 0\\
0.141732283464567 0\\
0.149606299212598 0\\
0.15748031496063 0\\
0.165354330708661 0\\
0.173228346456693 0\\
0.181102362204724 0\\
0.188976377952756 0\\
0.196850393700787 0\\
0.204724409448819 0\\
0.21259842519685 0\\
0.220472440944882 0\\
0.228346456692913 0\\
0.236220472440945 0\\
0.244094488188976 0\\
0.251968503937008 0\\
0.259842519685039 0\\
0.267716535433071 0\\
0.275590551181102 0\\
0.283464566929134 0\\
0.291338582677165 0\\
0.299212598425197 0\\
0.307086614173228 -0.00323248740472566\\
0.31496062992126 -0.0031581919468381\\
0.322834645669291 -0.00306948456035798\\
0.330708661417323 -0.00296610778064272\\
0.338582677165354 -0.00284784885105972\\
0.346456692913386 -0.0027145417202276\\
0.354330708661417 -0.00256606887965321\\
0.362204724409449 -0.00240236303409752\\
0.37007874015748 -0.00222340859791483\\
0.377952755905512 -0.00202924301114817\\
0.385826771653543 -0.00181995786968564\\
0.393700787401575 -0.00159569986432836\\
0.401574803149606 -0.00135667152418784\\
0.409448818897638 -0.00110313176041687\\
0.417322834645669 -0.000835396206881153\\
0.425196850393701 -0.000553837354997815\\
0.433070866141732 -0.000258884480597446\\
0.440944881889764 0\\
0.448818897637795 0\\
0.456692913385827 0\\
0.464566929133858 0\\
0.47244094488189 0\\
0.480314960629921 0\\
0.488188976377953 0\\
0.496062992125984 0\\
0.503937007874016 0\\
0.511811023622047 0\\
0.519685039370079 0\\
0.52755905511811 0\\
0.535433070866142 0.00450003588926954\\
0.543307086614173 0.00490759826436451\\
0.551181102362205 0.00531598816634518\\
0.559055118110236 0.00572404582641036\\
0.566929133858268 0.00613058668609008\\
0.574803149606299 0.0065344048418014\\
0.582677165354331 0.00693427661643088\\
0.590551181102362 0.00732896424784271\\
0.598425196850394 0.00771721968389853\\
0.606299212598425 0.00809778847316316\\
0.614173228346457 0.00846941374009439\\
0.622047244094488 0.00883084023317239\\
0.62992125984252 0.00918081843411695\\
0.637795275590551 0.00951810871606824\\
0.645669291338583 0.00984148553837087\\
0.653543307086614 0.0101497416654024\\
0.661417322834646 0.0104416923967255\\
0.669291338582677 0.0107161797957218\\
0.677165354330709 0.01097207690378\\
0.68503937007874 0.0112082919270688\\
0.692913385826772 0.0114237723829207\\
0.700787401574803 0.0116175091928888\\
0.708661417322835 0.0117885407096165\\
0.716535433070866 0.0119359566647763\\
0.724409448818898 0.0120589020254939\\
0.732283464566929 0.0121565807468703\\
0.740157480314961 0.0122282594084527\\
0.748031496062992 0.0122732707227861\\
0.755905511811024 0.0122910169044894\\
0.763779527559055 0.0122809728886602\\
0.771653543307087 0.012242689387803\\
0.779527559055118 0.012175795776906\\
0.78740157480315 0.0120800027967577\\
0.795275590551181 0.0119551050660958\\
0.803149606299213 0.0118009833937138\\
0.811023622047244 0.0116176068822184\\
0.818897637795276 0.0114050348157282\\
0.826771653543307 0.0111634183244297\\
0.834645669291339 0.0108930018195618\\
0.84251968503937 0.0105941241930818\\
0.850393700787402 0.0102672197769672\\
0.858267716535433 0.00991281905783811\\
0.866141732283465 0.00953154914332739\\
0.874015748031496 0.00912413397739495\\
0.881889763779528 0.00869139430255727\\
0.889763779527559 0.00823424736779797\\
0.897637795275591 0.00775370638172432\\
0.905511811023622 0.0072508797113407\\
0.913385826771654 0.00672696982761484\\
0.921259842519685 0.00618327199981116\\
0.929133858267717 0.0056211727413471\\
0.937007874015748 0.00504214801067728\\
0.94488188976378 0.00444776117140215\\
0.952755905511811 0.0038396607163847\\
0.960629921259842 0.00321957776104991\\
0.968503937007874 0.00258932331104145\\
0.976377952755906 0.00195078530855856\\
0.984251968503937 0.001305925458799\\
0.992125984251969 0.000656775829549233\\
1 0.000111337971270653\\
};
\end{axis}
\end{tikzpicture}%
        \caption{$\alpha = 10^{-3}$, $\beta = 10^{-8}$\label{fig:ell_38}}
    \end{subfigure}

    \begin{subfigure}[t]{0.475\textwidth}
        \begin{tikzpicture}

\begin{axis}[%
width=\textwidth,
xmin=0,
xmax=1,
ymin=-1.5,
ymax=1.5,
xlabel={$x_1$},
axis y line=left,
axis x line=middle,
legend style={legend cell align=left,align=left,draw=none,font=\footnotesize},
legend entries={$u_1$,$u_2$},
legend pos=north west
]
\addplot +[
line width=1,
]
table[row sep=crcr]{
0 0.00209132495762677\\
0.0078740157480315 0.0123478044154874\\
0.015748031496063 0.0246192415698687\\
0.0236220472440945 0.0369381128711569\\
0.031496062992126 0.0493202633576348\\
0.0393700787401575 0.061775852359287\\
0.047244094488189 0.0743094766572268\\
0.0551181102362205 0.0869203082190822\\
0.062992125984252 0.0996022463328522\\
0.0708661417322835 0.112344083889483\\
0.078740157480315 0.125129687495588\\
0.0866141732283465 0.137938191025912\\
0.094488188976378 0.15074420217461\\
0.102362204724409 0.163518021789787\\
0.110236220472441 0.176225876821393\\
0.118110236220472 0.188830168228604\\
0.125984251968504 0.20128973181255\\
0.133858267716535 0.213560110955264\\
0.141732283464567 0.225593854624392\\
0.149606299212598 0.237340842822665\\
0.15748031496063 0.248748603035672\\
0.165354330708661 0.259762604050105\\
0.173228346456693 0.270326537196971\\
0.181102362204724 0.280382614692871\\
0.188976377952756 0.289871889477791\\
0.196850393700787 0.29873458255645\\
0.204724409448819 0.306910410548084\\
0.21259842519685 0.314338913401094\\
0.220472440944882 0.320959782691693\\
0.228346456692913 0.326713191281397\\
0.236220472440945 0.331540136389498\\
0.244094488188976 0.335382809683824\\
0.251968503937008 0.338184912576019\\
0.259842519685039 0.339891846827952\\
0.267716535433071 0\\
0.275590551181102 0\\
0.283464566929134 0\\
0.291338582677165 0\\
0.299212598425197 0\\
0.307086614173228 0\\
0.31496062992126 0\\
0.322834645669291 0\\
0.330708661417323 0\\
0.338582677165354 0\\
0.346456692913386 0\\
0.354330708661417 0\\
0.362204724409449 0\\
0.37007874015748 0\\
0.377952755905512 0\\
0.385826771653543 0\\
0.393700787401575 0\\
0.401574803149606 0\\
0.409448818897638 0\\
0.417322834645669 0\\
0.425196850393701 0\\
0.433070866141732 0\\
0.440944881889764 -0.000230328627617862\\
0.448818897637795 -0.0319932069462927\\
0.456692913385827 -0.0649298992095749\\
0.464566929133858 -0.0989757950329746\\
0.47244094488189 -0.13406098264454\\
0.480314960629921 -0.170110384010492\\
0.488188976377953 -0.207043910487496\\
0.496062992125984 -0.244776638294138\\
0.503937007874016 -0.283219003929011\\
0.511811023622047 -0.322277017490166\\
0.519685039370079 -0.36185248789337\\
0.52755905511811 -0.401843263324895\\
0.535433070866142 -0.442143492716443\\
0.543307086614173 0\\
0.551181102362205 0\\
0.559055118110236 0\\
0.566929133858268 0\\
0.574803149606299 0\\
0.582677165354331 0\\
0.590551181102362 -0.723366330492653\\
0.598425196850394 -0.76198157790256\\
0.606299212598425 -0.799837289071046\\
0.614173228346457 0\\
0.622047244094488 -0.872770315330376\\
0.62992125984252 0\\
0.637795275590551 0\\
0.645669291338583 0\\
0.653543307086614 0\\
0.661417322834646 -0.700919246861546\\
0.669291338582677 -1.06047445361336\\
0.677165354330709 -1.08597391761455\\
0.68503937007874 -1.10952034864269\\
0.692913385826772 -1.13100898211699\\
0.700787401574803 -1.15033936422378\\
0.708661417322835 -1.1674157528541\\
0.716535433070866 -1.18214751074448\\
0.724409448818898 -1.1944494895051\\
0.732283464566929 -0.528232369294894\\
0.740157480314961 0\\
0.748031496062992 0\\
0.755905511811024 0\\
0.763779527559055 -1.21696280308044\\
0.771653543307087 -1.21325090902948\\
0.779527559055118 -1.14573562125971\\
0.78740157480315 0\\
0.795275590551181 -1.18495400193886\\
0.803149606299213 0\\
0.811023622047244 0\\
0.818897637795276 0\\
0.826771653543307 -1.10666072367074\\
0.834645669291339 -1.07988357405824\\
0.84251968503937 -1.05027807202157\\
0.850393700787402 -1.01788693358323\\
0.858267716535433 0\\
0.866141732283465 0\\
0.874015748031496 0\\
0.881889763779528 0\\
0.889763779527559 0\\
0.897637795275591 0\\
0.905511811023622 0\\
0.913385826771654 0\\
0.921259842519685 0\\
0.929133858267717 0\\
0.937007874015748 0\\
0.94488188976378 0\\
0.952755905511811 0\\
0.960629921259842 0\\
0.968503937007874 0\\
0.976377952755906 0\\
0.984251968503937 0\\
0.992125984251969 0\\
1 0\\
};
\addplot +[
line width=1,
]
table[row sep=crcr]{
0 0\\
0.0078740157480315 0\\
0.015748031496063 0\\
0.0236220472440945 0\\
0.031496062992126 0\\
0.0393700787401575 0\\
0.047244094488189 0\\
0.0551181102362205 0\\
0.062992125984252 0\\
0.0708661417322835 0\\
0.078740157480315 0\\
0.0866141732283465 0\\
0.094488188976378 0\\
0.102362204724409 0\\
0.110236220472441 0\\
0.118110236220472 0\\
0.125984251968504 0\\
0.133858267716535 0\\
0.141732283464567 0\\
0.149606299212598 0\\
0.15748031496063 0\\
0.165354330708661 0\\
0.173228346456693 0\\
0.181102362204724 0\\
0.188976377952756 0\\
0.196850393700787 0\\
0.204724409448819 0\\
0.21259842519685 0\\
0.220472440944882 0\\
0.228346456692913 0\\
0.236220472440945 0\\
0.244094488188976 0\\
0.251968503937008 0\\
0.259842519685039 0\\
0.267716535433071 -0.340451597553905\\
0.275590551181102 -0.339818574000443\\
0.283464566929134 -0.337940181149598\\
0.291338582677165 -0.334771792413643\\
0.299212598425197 -0.330272003790229\\
0.307086614173228 -0.324402929850415\\
0.31496062992126 -0.317130436434062\\
0.322834645669291 -0.30842438738189\\
0.330708661417323 -0.298258885551307\\
0.338582677165354 -0.286612491086727\\
0.346456692913386 -0.273468420427287\\
0.354330708661417 -0.258814728550688\\
0.362204724409449 -0.242644474344938\\
0.37007874015748 -0.224955869082364\\
0.377952755905512 -0.205752411050686\\
0.385826771653543 -0.185043009724172\\
0.393700787401575 -0.1628420793183\\
0.401574803149606 -0.139169586929569\\
0.409448818897638 -0.11405111000584\\
0.417322834645669 -0.087517871944666\\
0.425196850393701 -0.0596067545302618\\
0.433070866141732 -0.0303602923276237\\
0.440944881889764 0\\
0.448818897637795 0\\
0.456692913385827 0\\
0.464566929133858 0\\
0.47244094488189 0\\
0.480314960629921 0\\
0.488188976377953 0\\
0.496062992125984 0\\
0.503937007874016 0\\
0.511811023622047 0\\
0.519685039370079 0\\
0.52755905511811 0\\
0.535433070866142 0\\
0.543307086614173 0.482645532972274\\
0.551181102362205 0.523234149689461\\
0.559055118110236 0.563794160344938\\
0.566929133858268 0.604207947410813\\
0.574803149606299 0.644355786113347\\
0.582677165354331 0.684116178469197\\
0.590551181102362 0\\
0.598425196850394 0\\
0.606299212598425 0\\
0.614173228346457 0.836808506390636\\
0.622047244094488 0\\
0.62992125984252 0.907598415607822\\
0.637795275590551 0.941169480953777\\
0.645669291338583 0.973361578896138\\
0.653543307086614 1.00405458894639\\
0.661417322834646 0.332222875863484\\
0.669291338582677 0\\
0.677165354330709 0\\
0.68503937007874 0\\
0.692913385826772 0\\
0.700787401574803 0\\
0.708661417322835 0\\
0.716535433070866 0\\
0.724409448818898 0\\
0.732283464566929 0.676013355605582\\
0.740157480314961 1.21145319126152\\
0.748031496062992 1.2160153665522\\
0.755905511811024 1.2178693524967\\
0.763779527559055 0\\
0.771653543307087 0\\
0.779527559055118 0.0609657960842675\\
0.78740157480315 1.19727123018487\\
0.795275590551181 0\\
0.803149606299213 1.1697330214053\\
0.811023622047244 1.15160510557378\\
0.818897637795276 1.13057603255764\\
0.826771653543307 0\\
0.834645669291339 0\\
0.84251968503937 0\\
0.850393700787402 0\\
0.858267716535433 0.982763371656226\\
0.866141732283465 0.944972408590676\\
0.874015748031496 0.904585963576701\\
0.881889763779528 0.861686345805346\\
0.889763779527559 0.816365249942505\\
0.897637795275591 0.768723718960987\\
0.905511811023622 0.718872088661913\\
0.913385826771654 0.666929904740665\\
0.921259842519685 0.613025812597973\\
0.929133858267717 0.5572974206858\\
0.937007874015748 0.499891138097463\\
0.94488188976378 0.440961987086862\\
0.952755905511811 0.380673391199727\\
0.960629921259842 0.31919693969762\\
0.968503937007874 0.256712128926592\\
0.976377952755906 0.193406081177756\\
0.984251968503937 0.129473241287118\\
0.992125984251969 0.0651150503532205\\
1 0.0110386858022917\\
};
\end{axis}
\end{tikzpicture}%
        \caption{$\alpha = 10^{-5}$, $\beta = 10^{-3}$\label{fig:ell_53}}
    \end{subfigure}
    \hfill
    \begin{subfigure}[t]{0.475\textwidth}
        \input{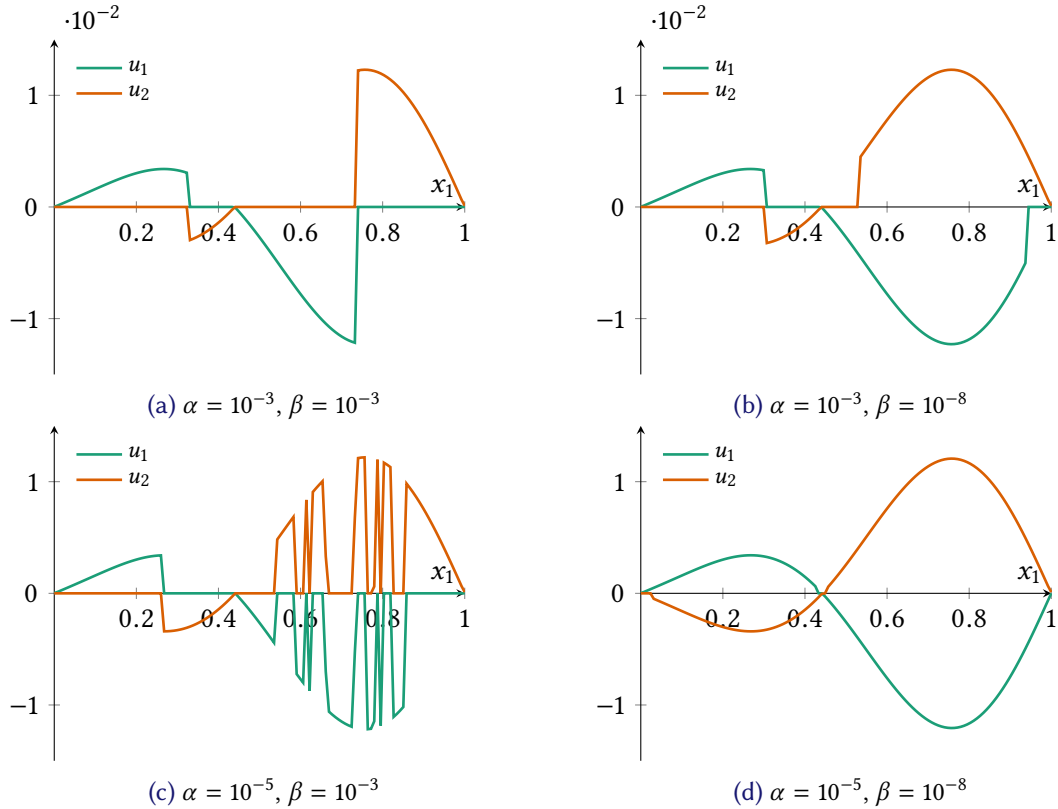}
        \caption{$\alpha = 10^{-5}$, $\beta = 10^{-8}$\label{fig:ell_58}}
    \end{subfigure}
    \caption{Elliptic problem, effect of $\alpha,\beta$ on structure of control $u_\gamma$ (left: switching, right: no switching)}
    \label{fig:elliptic}
\end{figure}
We begin by illustrating the effects of the values of $\alpha$ and $\beta$ on the structure of the resulting controls. 
\Cref{fig:elliptic} shows the final computed controls $u_\gamma$ for the same target $z$ and different combinations of control costs. 
For the choice  $\alpha=\beta=10^{-3}$ (\cref{fig:ell_33}), the control has a pure switching structure, with $80$ nodes (out of $128$) having values in the active set $Q^\gamma_1$ and $48$ nodes in the set $Q^\gamma_2$ (the remaining sets being empty); in particular, the singular arc $\calS$ is empty. 
Furthermore, the effect of the $L^2$ costs on the active control components can be observed clearly. 
Decreasing $\beta$ to $10^{-8}$ results in a control that is no longer purely switching (\cref{fig:ell_38}), although some switching behavior still obtains in parts of $D$; the resulting active sets have $51$ nodes in $Q^\gamma_1$, $25$ nodes in $Q^\gamma_2$, and $52$ nodes in the regularized free arc $Q^\gamma_0$. 
Since $\alpha$ is unchanged, the magnitude of the active controls is the same as before. 
Decreasing $\alpha$, on the other hand, allows for controls of larger magnitude, but results in the appearance of singular arcs. 
For $\alpha=10^{-5}$ and $\beta=10^{-3}$ (\cref{fig:ell_53}), we observe a control which is almost purely switching ($66$ and $59$ nodes in $Q^\gamma_1$ and $Q^\gamma_2$, respectively) but still has a non-negligible singular arc with $3$ nodes in $Q^\gamma_{12}$. 
The control shows a chittering behavior on part of the switching arc, which can be attributed to the weak but not pointwise convergence of the regularized controls. 
For the smaller value of $\beta$ (\cref{fig:ell_58}), the singular arc disappears at the expense of the appearance of a large free arc ($5$ nodes in $Q^\gamma_1$, $3$ nodes in $Q^\gamma_2$, and $120$ nodes in~$Q^\gamma_0$).

Let us briefly comment on the convergence behavior of the ``globalized'' Newton method.
For $\gamma>10^{-9}$, the semismooth Newton iteration shows the typical superlinear behavior, converging within two or three (full) steps to a solution of the system \eqref{eq:opt_switching_reg}.
For smaller values of $\gamma$, backtracking becomes necessary after one full step, but, depending on the presence of singular arcs, often enters into a superlinear phase again where full steps are taken to convergence.
Specifically, in the case of $\alpha=\beta=10^{-3}$, the iteration terminates successfully at $\gamma=10^{-12}$ with only a few reduced steps necessary.
For $\alpha = 10^{-5}$ and $\beta = 10^{-3}$, more line searches are performed, but the final superlinear phase is still observed for $\gamma>10^{-13}$, after which the Newton iteration terminated since no sufficient decrease in the residual was possible.
However, restarting with smaller $\gamma$ still allowed some successful steps before terminating again, which continued until the specified terminal value of $\gamma=10^{-16}$ was reached.
For $\beta = 10^{-8}$, no backtracking was necessary, and the algorithm showed the typical behavior of a semismooth Newton method with continuation (terminating successfully at $\gamma=10^{-9}$ for $\alpha = 10^{-3}$ and at $\gamma=10^{-10}$ for $\alpha=10^{-5}$).

\bigskip

To demonstrate the applicability of the proposed approach to switching control of parabolic equations, we also show results for the one-dimensional heat equation, where $S:u\mapsto y$ satisfying
\begin{equation}
    y_t - \Delta y = Bu = \chi_{\omega_1}(x)u_1(t) + \chi_{\omega_2}(x)u_2(t)
\end{equation}
with $\Omega = [-1,1]$, $D=[0,2]$, $\Omega_T = D\times \Omega$, 
\begin{equation}
    \omega_1 = \set{x\in \Omega}{x< -\tfrac12},\qquad
    \omega_2 = \set{x\in \Omega}{x> \tfrac12}.
\end{equation}
As a target, we choose the trajectory of the heat equation with the right-hand side
\begin{equation}
    f(t,x) = \begin{cases} 63 & \text{if } |t-1-x|<\tfrac1{10},\\
        0 & \text{otherwise},
    \end{cases}
\end{equation}
see \cref{fig:par_target}.
\begin{figure}[t]
    \centering
    \begin{tikzpicture}
        \node[anchor=south west,inner sep=0] at (0,0) {\includegraphics[width=0.5\textwidth]{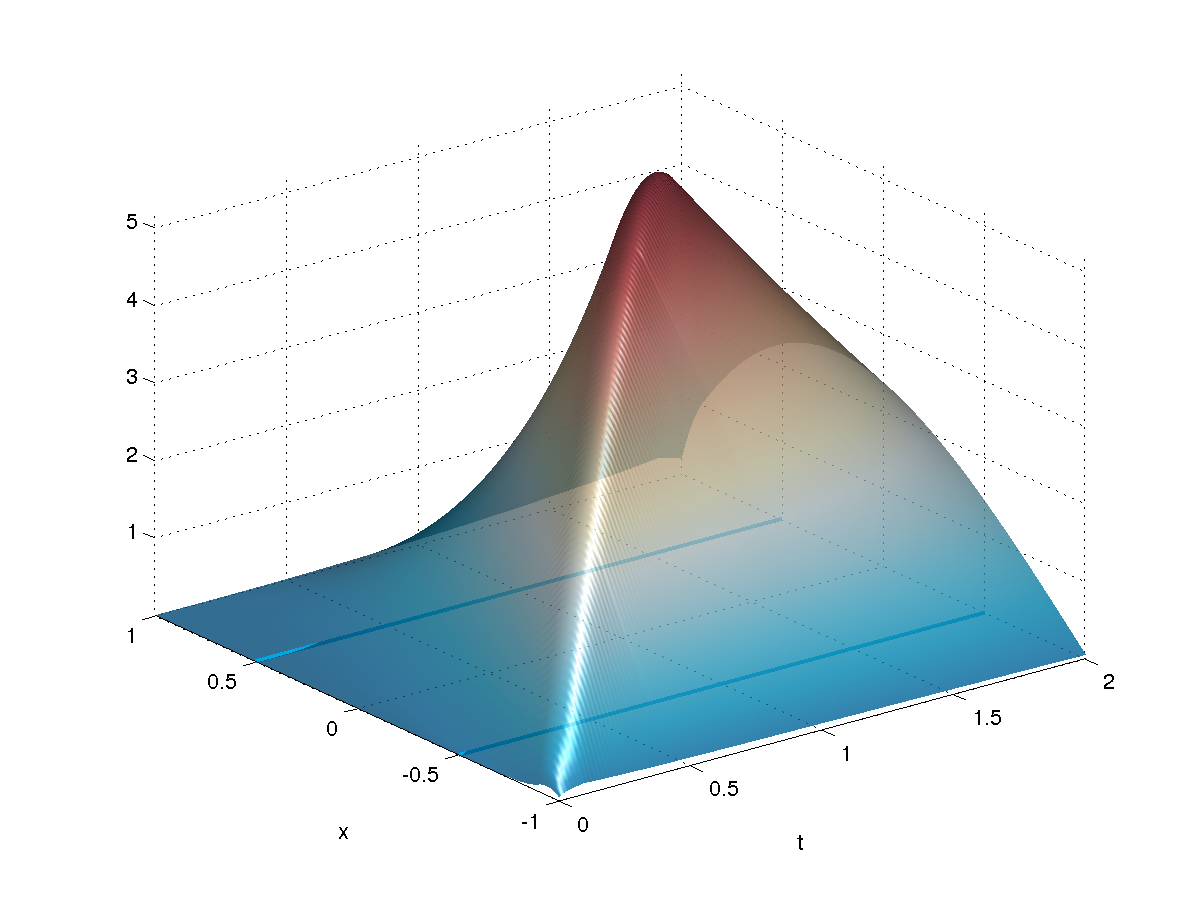}};
        \node[color=white] at (1.6,1.7) {\footnotesize $\omega_2$};
        \node[color=white] at (6,1.55) {\footnotesize $\omega_1$};
    \end{tikzpicture}
    \caption{Parabolic problem, target $z$ and control domains $\omega_1$, $\omega_2$}
    \label{fig:par_target}
\end{figure}
The discretization is similar as in the elliptic case, using a full space-time discontinuous Galerkin discretization corresponding to a backward Euler method with $N_h=128$ spatial grid points and $N_t = 512$ time steps.

The resulting controls for $\alpha=10^{-1}$ are shown in \cref{fig:parabolic}. 
For $\beta = 1$ (\cref{fig:par_10}), the control is again of purely switching type with $256$ nodes each in $Q^\gamma_1$ and $Q^\gamma_2$. 
No backtracking was necessary, and the continuation terminated successfully at $\gamma=10^{-9}$. 
The control for $\beta = 10^{-1}$ (\cref{fig:par_11}) shows a free arc, with $77$ nodes in $Q^\gamma_1$, $110$ nodes in $Q^\gamma_2$, and $325$ nodes in $Q^\gamma_0$. 
The convergence behavior is now different due to the intermittent appearance of singular arcs: Although the first continuation step with $\gamma=10^{-2}$ shows the usual superlinear convergence with full steps, the resulting iterate contains nodes in $Q^\gamma_{10}$ and $Q^\gamma_{20}$. 
Subsequently, the iterations for $\gamma>10^{-5}$ suffer from progressively smaller steps until no sufficient decrease is possible. 
At $\gamma=10^{-5}$, however, the corresponding singular arc $\partial\calI$ is empty and the iteration returns to superlinear convergence with full steps, terminating successfully at $\gamma=10^{-9}$. 
The difference to the elliptic case can be attributed to the lower regularity of the adjoint state $p$ with respect to the control dimension (here: time) and the corresponding smaller norm gap in the regularized subdifferential~$H_\gamma(p)$.
\begin{figure}
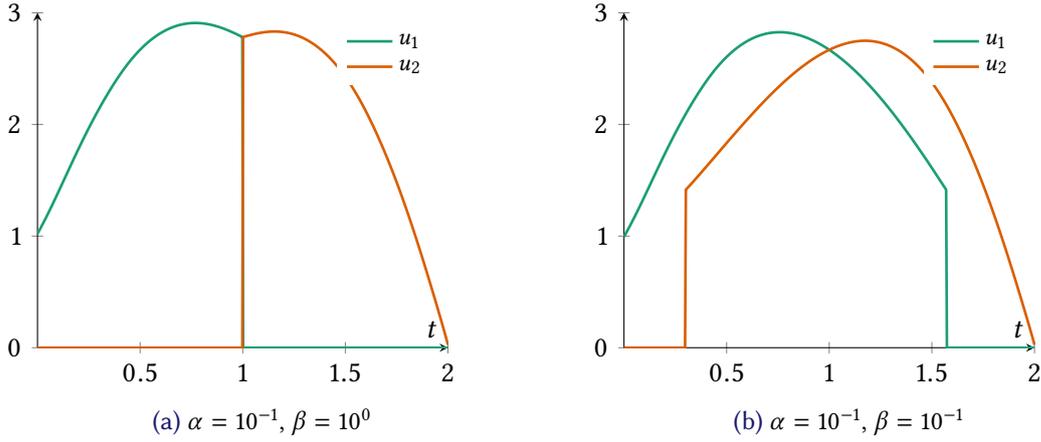

    \centering
    \begin{subfigure}[t]{0.475\textwidth}
        \input{par_10_control.tikz}
        \caption{$\alpha = 10^{-1}$, $\beta = 10^{0}$\label{fig:par_10}}
    \end{subfigure}
    \hfill
    \begin{subfigure}[t]{0.475\textwidth}
        \input{par_11_control.tikz}
        \caption{$\alpha = 10^{-1}$, $\beta = 10^{-1}$\label{fig:par_11}}
    \end{subfigure}
    \caption{Parabolic problem, effect of $\alpha,\beta$ on structure of control $u_\gamma$ (left: switching, right: no switching)}
    \label{fig:parabolic}
\end{figure}

\section{Conclusion}

A framework for optimal control problems was presented that promotes controls of switching type. While switching is promoted by a sparsity-enhancing part of the cost  functional, the active controls are weighted with quadratic cost. Analysis of the proposed approach is carried out by techniques from convex analysis, while its numerical solution is achieved using a semismooth Newton method with continuation and line searches. Numerical results support the theoretical~findings. 

There are many interesting follow-up topics, including the treatment of problems with nonlinear control-to-state mappings, a more detailed analysis of the influence of the control cost parameters on the structure of the controls, and problems with multiple controls exhibiting generalized switching structures.

\appendix

\section{Application to other binary penalties}\label{sec:binary}

This appendix demonstrates the application of the approach of \cref{sec:pointwise} to other functionals involving the binary functional $|v|_0$. While the Fenchel conjugates and subdifferentials have already been obtained in the previous works cited below, the proximal mappings and corresponding Moreau--Yosida regularizations and complementarity formulations are new.

\subsection{Sparse control}
\label{sec:sparse}

We first consider the functional
\begin{equation}
    \calG(u) = \frac\alpha2\norm{u}_{L^2}^2 + \beta \int_\Omega |u(x)|_0\,dx,
\end{equation}
which promotes sparsity in optimal control and, contrary to $L^1$-type penalties, allows separate penalization of magnitude and support; see \cite{IK:2012}.
Setting
\begin{equation}
    g(v) = \frac\alpha2 v^2 + \beta |v|_0 := 
    \begin{cases}  \frac\alpha2 v^2 + \beta & \text{if }v\neq 0,\\
        0 & \text{if }v = 0,
    \end{cases}
\end{equation}
we compute the Fenchel conjugate
\begin{equation}\label{eq:sparse:fenchel_pt}
    g^*(q) = \sup_{v\in\R}v\cdot q-g(v)\tag{\textsc{a}.1}
\end{equation}
by case distinction. Assume that the supremum is attained for some $\bar v\in\R$. Then we discriminate the following two cases:
\begin{enumerate}[(i)]
    \item $\bar v=0$, in which case $g(\bar v)=0$ and hence $g^*(q) = 0$;
    \item $\bar v\neq 0$, in which case $g(\bar v) = \frac\alpha2 \bar v^2+\beta$. Since $g$ is differentiable at $\bar v$, the necessary condition for $\bar v$ to attain the maximum is $q = \alpha \bar v$. Solving for $\bar v$ and inserting in~\eqref{eq:sparse:fenchel_pt} yields
        \begin{equation}
            g^*(q) = \frac1{2\alpha} q^2-\beta.
        \end{equation}
\end{enumerate}
It remains to decide which of these cases is attained for a given $q$, i.e., whether
\begin{equation}
    g_0^*(q) := 0 < \frac1{2\alpha}q^2 - \beta =: g_1^*(q).
\end{equation}
This directly yields
\begin{equation}
    g^*(q) = \max_{i\in\{0,1\}} g_i^*(q) = 
    \begin{cases}
        0 & \text{if }|q| \le  \sqrt{2\alpha\beta},\\
        \text{if } \frac1{2\alpha} q^2 - \beta & \text{if }|q| > \sqrt{2\alpha\beta}.
    \end{cases}
\end{equation}
as well as 
\begin{equation}\label{eq:l0_subdiff}
    \partial g^*(q) = \overline{\mathrm{co}} \left(\bigcup_{\{i:g^*(q)= g_i^*(q)\}}\left\{ (g_{i}^*)'(q)\right\}\right)
    = 
    \begin{cases}
        0 &\text{if } |q| < \sqrt{2\alpha\beta},\\
        \left[0,\tfrac1\alpha q\right] & \text{if }|q| = \sqrt{2\alpha\beta},\\
        \frac1\alpha q  &\text{if } |q| > \sqrt{2\alpha\beta}.
    \end{cases}\tag{\textsc{a}.2}
\end{equation}

We now turn to the computation for given $\gamma>0$ and $v\in\R$ of the proximal mapping  $w = \prox_{\gamma g^*}(v)$ of $g^*$ or, equivalently, the resolvent of $\partial g^*$, which is characterized by the relation $v\in (\mathrm{Id}+\gamma\partial g^*)(w)$.
We now distinguish all possible cases in~\eqref{eq:l0_subdiff}:
\begin{enumerate}[(i)]
    \item $|w|<\sqrt{2\alpha\beta}$: In this case $v=w$, which implies that $|v|<\sqrt{2\alpha\beta}$.
    \item $|w|> \sqrt{2\alpha\beta}$: In this case $v=(1+\frac\gamma\alpha)w$, which implies that $|v|>(1+\frac\gamma\alpha)\sqrt{2\alpha\beta}$.
    \item $|w| = \sqrt{2\alpha\beta}$: In this case $v\in[w,(1+\frac\gamma\alpha)w]$, which implies that $\sqrt{2\alpha\beta}\leq |v| \leq (1+\frac\gamma\alpha)\sqrt{2\alpha\beta}$.
\end{enumerate}
Inserting this into the definition of the Moreau--Yosida regularization and simplifying yields
\begin{equation}\label{eq:sparse:subdiff_gamma}
    (\partial g^*)_\gamma(q) = 
    \begin{cases}
        0 &\text{if } |q| < \sqrt{2\alpha\beta},\\
        \tfrac1\gamma\left(q-\sqrt{2\alpha\beta}\sign(q)\right) &\text{if }|q|\in\left[\sqrt{2\alpha\beta},  (1+\tfrac\gamma\alpha)\sqrt{2\alpha\beta}\right],\\
        \tfrac{1}{\alpha+\gamma}q & \text{if }|q| > (1+\tfrac\gamma\alpha)\sqrt{2\alpha\beta},
    \end{cases}
\end{equation}
which can be interpreted as a soft-thresholding operator.

Since $h_\gamma := (\partial g^*)_\gamma$ is Lipschitz continuous and piecewise differentiable, it is semismooth, and its Newton-derivative at $q$ in direction $\delta q$ is given by
\begin{equation}
    D_N h_\gamma(q)\partial q = 
    \begin{cases}
        0 &\text{if } |q| < \sqrt{2\alpha\beta},\\
        \tfrac1\gamma\delta q &\text{if }|q|\in\left[\sqrt{2\alpha\beta},  (1+\tfrac\gamma\alpha)\sqrt{2\alpha\beta}\right],\\
        \tfrac{1}{\alpha+\gamma}\delta q & \text{if }|q| > (1+\tfrac\gamma\alpha)\sqrt{2\alpha\beta}.
    \end{cases}
\end{equation}

\subsection{Multi-bang control}\label{sec:multibang}

We now consider the \emph{multi-bang} functional
\begin{equation}
    g(v) = \frac{\alpha}{2} v^2 + \beta\prod_{i=1}^d |v-u_i|_0 + \delta_{[u_1,u_d]}(v),
\end{equation}
where $u_1,\dots,u_d$ are given desired control states and $\delta_C$ denotes the indicator function of the convex set $C$. In optimal control problems, the binary term (together with the pointwise constraints) promotes controls which, for $\beta$ sufficiently large, take on only the desired values almost everywhere except possibly on a singular set; see \cite{CK:2013}.

Proceeding as in \cref{sec:sparse} yields the Fenchel conjugate
\begin{equation}\label{eq:multibang:conj}
    g^*(q) =
    \begin{cases}
        qu_1 - \frac{\alpha}{2} u_1^2  &\text{if } q-\alpha u_1 \leq \sqrt{2\alpha\beta} \quad\text{and}\quad q\leq\frac\alpha2(u_1+u_{2}),\\
        qu_i - \frac{\alpha}{2} u_i^2 &\text{if } |q-\alpha u_i| \leq \sqrt{2\alpha\beta} \quad\text{and}\quad \frac\alpha2(u_{i-1}+u_i)\leq q \leq \frac\alpha2(u_i+u_{i+1}), 1<i<d,\\
        qu_d - \frac{\alpha}{2} u_d^2 &\text{if } q-\alpha u_d \geq \sqrt{2\alpha\beta} \quad\text{and}\quad \frac\alpha2(u_{d}+u_{d-1}) \leq q,\\
        \frac{1}{2\alpha} q^2 - \beta &\text{if } |q-\alpha u_j| \leq \sqrt{2\alpha\beta} \ \text{ for all }\ j\in\{1,\dots,d\}\quad\text{and}\quad \alpha u_1 \leq q \leq \alpha u_d,
    \end{cases}
\end{equation}
whose subdifferential is
\begin{equation}
    \partial g^*(q)=
    \begin{cases}
        \{u_i\} & \text{if } q \in Q_i,\ 1\leq i < d,\\
        \{\frac1\alpha q\} &\text{if } q \in Q_0,\\
        \left[u_i,\frac1\alpha q\right] & \text{if } q \in Q_{i0},\ 1\leq i\leq d,\\
        [u_i,u_{i+1}] & \text{if } q \in Q_{i,i+1},\ 1\leq i < d,
    \end{cases}
\end{equation}
where
\begin{align}\label{eq:multibang:def_pi}
    Q_1 &= \set{q}{q-\alpha u_1 < \sqrt{2\alpha\beta} \quad\text{and}\quad q<\tfrac\alpha2(u_1+u_{2})},\\
    Q_i &= \set{q}{|q-\alpha u_i| < \sqrt{2\alpha\beta} \quad\text{and}\quad \tfrac\alpha2(u_{i-1}+u_i)<q< \tfrac\alpha2(u_i+u_{i+1})}\quad \text{ for } 1<i<d,\\
    Q_d &= \set{q}{q-\alpha u_d > \sqrt{2\alpha\beta} \quad\text{and}\quad \tfrac\alpha2(u_{d}+u_{d-1}) < q},\\
    Q_0 &= \set{q}{|q-\alpha u_j| > \sqrt{2\alpha\beta}\quad\text{for all } j\in\{1,\dots,d\}\quad\text{and}\quad \alpha u_1 < q < \alpha u_d}\\
    Q_{i0} &= \set{q}{|q-\alpha u_i| = \sqrt{2\alpha\beta}}\quad\text{for } 1\leq i \leq d,\\
    Q_{i,i+1} &= \set{q}{q = \tfrac\alpha2(u_{i}+u_{i+1})} \quad\text{for } 1 \leq i < d,
\end{align}
Note that some of these sets can be empty. In fact, for $\beta$ sufficiently large, $Q_0$ and hence $Q_{i0}$, $i=1,\dots,d$, can be guaranteed to vanish; see \cite[\S\,2.3]{CK:2013}.

To compute for given $\gamma>0$ and $v\in\R$ the resolvent $w=(\mathrm{Id}+\gamma\partial g^*)^{-1}(v)$ of $\partial g^*$, we again use the relation $v \in\{w\} + \gamma\partial g^*(w)$
and follow the case differentiation in the subdifferential.
\begin{enumerate}[(i)]
    \item $w\in Q_i$ for some $i\in\{1,\dots,d\}$: In this case, $v = w + \gamma u_i$, which implies that
        \begin{equation}
            |v-(\alpha+\gamma)u_i| \leq \sqrt{2\alpha\beta}
        \end{equation}
        and
        \begin{equation}
            \tfrac\alpha2\left(u_{i-1} + \left(1 + \tfrac{2\gamma}{\alpha}\right)u_i\right) < v < \tfrac\alpha2\left(\left(1+ \tfrac{2\gamma}{\alpha}\right)u_i+u_{i+1}\right)
        \end{equation}
        (with the first and last condition being void for $i=1$ and $i=d$, respectively).

    \item $w\in Q_0$: In this case, $v = \left(1+\frac\gamma\alpha\right)w$, which implies that
        \begin{equation}
            |\tfrac\alpha{\alpha+\gamma} v - \alpha u_j| > \sqrt{2\alpha\beta}\quad\text{ for all } j\in\{1,\dots,d\}
        \end{equation}
        and
        \begin{equation}
            (\alpha+\gamma)u_1 < v < (\alpha + \gamma u_d).
        \end{equation}

    \item $w\in Q_{i0}$ for some $i\in\{1,\dots,d\}$: In this case, $v\in[w,(1+\frac\gamma\alpha)w]$ and $w = \alpha u_i + \sqrt{2\alpha\beta}$, which implies that
        \begin{equation}
            \sqrt{2\alpha\beta} \leq  v - (\alpha + \gamma) u_i \leq \left(1+\frac\gamma\alpha\right)\sqrt{2\alpha\beta}. 
        \end{equation}

    \item $w\in Q_{i,i+1}$ for some $i\in\{1,\dots,d-1\}$: In this case, $v\in[w+\gamma u_i,w+\gamma u_{i+1}]$ and $w = \frac\alpha2 (u_i +u_{i+1})$, which implies that
        \begin{equation}
            \tfrac\alpha2\left(\left(1 + \tfrac{2\gamma}{\alpha}\right)u_i+u_{i+1}\right) \leq v \leq \tfrac\alpha2\left(u_i+\left(1+ \tfrac{2\gamma}{\alpha}\right)u_{i+1}\right).
        \end{equation}
\end{enumerate}
Inserting this into the definition of the Moreau--Yosida regularization and simplifying, we obtain
\begin{equation}\label{eq:multibang:subdiff_gamma}
    (\partial g^*)_\gamma(q) = 
    \begin{cases}
        u_i & \text{if } q \in Q_i^\gamma \quad\text{for some }i\in\{1,\dots,d\},\\
        \tfrac1{\alpha+\gamma} q & \text{if } q \in Q_0^\gamma,\\
        \tfrac1\gamma\left(q-(\alpha u_i + \sqrt{2\alpha\beta})\right) & \text{if } q \in Q_{i0}^\gamma \quad\text{for some }i\in\{1,\dots,d\},\\
        \tfrac1\gamma\left(q-\tfrac\alpha2(u_i+u_{i+1})\right) & \text{if } q \in Q_{i,i+1}^\gamma \quad\text{for some }i\in\{1,\dots,d-1\},
    \end{cases}
\end{equation}
where
\begin{align}
    Q_1^\gamma &= \set{q}{q-(\alpha+\gamma)u_1 < \sqrt{2\alpha\beta} \quad\text{and}\quad  q< \tfrac\alpha2\left(\left(1+ \tfrac{2\gamma}{\alpha}\right)u_1+u_2\right)},\\
    Q_i^\gamma &= \big\{q:|q-(\alpha+\gamma) u_i| < \sqrt{2\alpha\beta}\quad\text{and}\\ 
    \MoveEqLeft[-3] \tfrac\alpha2\left(u_{i-1} + \left(1 + \tfrac{2\gamma}{\alpha}\right)u_i\right) < q < \tfrac\alpha2\left(\left(1+ \tfrac{2\gamma}{\alpha}\right)u_i+u_{i+1}\right)\big\}
    \quad \text{ for } 1<i<d,\\
    Q_d^\gamma &= \set{q}{q-(\alpha+\gamma)u_d > \sqrt{2\alpha\beta} \quad\text{and}\quad \tfrac\alpha2\left(u_{d-1} + \left(1 + \tfrac{2\gamma}{\alpha}\right)u_d\right) < q},\\
    Q_0^\gamma &= \set{q}{|q-(\alpha+\gamma) u_j| > \sqrt{2\alpha\beta}\ \text{ for all } j\in\{1,\dots,d\}\ \text{ and }\ (\alpha+\gamma) u_1 < q < (\alpha+\gamma) u_d},\\
    Q_{i0}^\gamma &= \set{q}{\sqrt{2\alpha\beta} \leq  q - (\alpha + \gamma) u_i \leq \left(1+\frac\gamma\alpha\right)\sqrt{2\alpha\beta}}\quad\text{for } 1\leq i \leq d,\\
    Q_{i,i+1}^\gamma &= \set{q}{\tfrac\alpha2\left(\left(1 + \tfrac{2\gamma}{\alpha}\right)u_i+u_{i+1}\right) \leq q \leq \tfrac\alpha2\left(u_i+\left(1+ \tfrac{2\gamma}{\alpha}\right)u_{i+1}\right)} \quad\text{for } 1 \leq i < d.
\end{align}

Since $h_\gamma := (\partial g^*)_\gamma$ is Lipschitz continuous and piecewise differentiable, it is semismooth, and its Newton-derivative at $q$ in direction $\delta q$ is given by
\begin{equation}
    D_N h_\gamma(q)\delta q = 
    \begin{cases}
        0 & \text{if } q \in Q_i^\gamma \quad\text{for some }i\in\{1,\dots,d\},\\
        \tfrac1{\alpha+\gamma} \delta q & \text{if } q \in Q_0^\gamma,\\
        \tfrac1\gamma\delta q & \text{if } q \in Q_{i0}^\gamma \quad\text{for some }i\in\{1,\dots,d\},\\
        \tfrac1\gamma\delta q & \text{if } q \in Q_{i,i+1}^\gamma \quad\text{for some }i\in\{1,\dots,d-1\}.
    \end{cases}
\end{equation}

\section{Biconjugate of \texorpdfstring{$\scriptstyle g$}{g}}\label{sec:biconjugate}

We now compute the biconjugate $g^{**}$ used in \cref{thm:singular_set}.
As in \cref{sec:switching:conjugate}, we proceed by a casewise maximization based on the definition of $g^*$; however, we need to take into account the restrictions $q\in Q_i$. We assume that $v_1,v_2\geq 0$, the remaining cases following by symmetry. Consider first 
\begin{equation}
    g^{**}_1(v) = \sup_{q\in Q_1} v\cdot q - \frac{1}{2\alpha} q_1^2
\end{equation}
and note that the supremum can only be attained for $q_1,q_2\geq 0$.
Introducing Lagrange multipliers $\lambda,\mu\geq 0$ for the constraints
$q_1-q_2\geq 0$ and $\sqrt{2\alpha\beta}-q_2\geq 0$, we obtain the KKT system
\begin{equation}
    \left\{\begin{aligned}
            v_1 - \frac1\alpha \bar q_1 + \bar\lambda &= 0,\\
            v_2 -\bar\lambda - \bar\mu &=0,\\
            \bar\lambda(\bar q_1-\bar q_2) &=0,\\
            \bar\mu\left(\sqrt{2\alpha\beta}-\bar q_2\right) &=0.
    \end{aligned}\right.
\end{equation}
We now make a case differentiation based on the optimal value of the multipliers $\bar\lambda,\bar\mu$.
\begin{enumerate}[(i)]
    \item $\bar\mu=0$: Adding the first two equations then yields
        \begin{equation}
            v_1 + v_2 = \frac1\alpha \bar q_1.
        \end{equation}
        To obtain an equation for $\bar q_2$, we further discriminate based on the value of $\bar \lambda$:
        \begin{enumerate}
            \item $\bar \lambda=0$: The second equation yields the condition $v_2 = 0$. In this case, the value of $\bar q_2$ is irrelevant to the supremum and we obtain for any admissible $\bar q_2$ 
                \begin{equation}
                    g^{**}_1(v) = \frac\alpha2 v_1^2.
                \end{equation}
            \item $\bar \lambda\neq 0$: In this case, $\bar q_1 = \bar q_2 = \alpha(v_1+v_2)$ and we obtain
                \begin{equation}
                    g^{**}_1(v) = \frac\alpha2 (v_1+v_2)^2,
                \end{equation}
                while the condition $\bar q_2\leq \sqrt{2\alpha\beta}$ translates into
                \begin{equation}
                    v_1+v_2 \leq \sqrt{\frac{2\beta}\alpha}.
                \end{equation}
        \end{enumerate}
    \item $\mu\neq 0$: This implies $\bar q_2=\sqrt{2\alpha\beta}$.  For the value of $\bar q_1$, we again further discriminate based on the value of $\bar \lambda$:
        \begin{enumerate}
            \item $\bar \lambda = 0$: The first equation then yields $v_1 = \frac1\alpha \bar q_1$ and we obtain
                \begin{equation}
                    g^{**}_1(v) = \frac\alpha2 v_1^2 + \sqrt{2\alpha\beta}v_2,
                \end{equation}
                while the condition $\bar q_1 \geq \bar q_2=\sqrt{2\alpha\beta}$ translates into
                \begin{equation}
                    v_1 \geq \sqrt{\frac{2\beta}\alpha}.
                \end{equation}
            \item $\bar \lambda\neq 0$: In this case, $\bar q_1=\bar q_2=\sqrt{2\alpha\beta}$, which yields
                \begin{equation}
                    g^{**}_1(v) =\sqrt{2\alpha\beta}(v_1+v_2) - \beta.
                \end{equation}
                Note that no conditions on $v_1,v_2$ are obtained.
        \end{enumerate}
\end{enumerate}
Collecting these cases, we obtain
\begin{equation}
    g^{**}_1(v) \in \begin{cases}  
        \frac\alpha2 (v_1+v_2)^2 & \text{ if } v_1+v_2 \leq \sqrt{\tfrac{2\beta}\alpha},\\
        \frac\alpha2 v_1^2 + \sqrt{2\alpha\beta}v_2 & \text{ if }  v_1 \geq \sqrt{\tfrac{2\beta}\alpha},\\
        \sqrt{2\alpha\beta}(v_1+v_2) - \beta,
    \end{cases}
\end{equation}

We proceed similarly for
\begin{equation}
    g^{**}_2(v) = \sup_{q\in Q_2} v\cdot q - \frac{1}{2\alpha} q_2^2
\end{equation}
to obtain the possible values and conditions
\begin{equation}
    g^{**}_2(v) \in \begin{cases}  
        \frac\alpha2 (v_1+v_2)^2 & \text{ if } v_1+v_2 \leq \sqrt{\tfrac{2\beta}\alpha},\\
        \frac\alpha2 v_2^2 + \sqrt{2\alpha\beta}v_1 & \text{ if }  v_2 \geq \sqrt{\tfrac{2\beta}\alpha},\\
        \sqrt{2\alpha\beta}(v_1+v_2) - \beta,
    \end{cases}
\end{equation}
where the case (i)\,a) has been absorbed into the first and second case (which for $v_1=0$ are exhaustive).

For 
\begin{equation}
    g^{**}_0(v) = \sup_{q\in Q_0} v\cdot q - \frac{1}{2\alpha} (q_1^2+q_2^2) +\beta,
\end{equation}
we use the fact that the optimality conditions for the maximizer are given by $\bar q = P_{Q_0}(\alpha v)$, where $P_{Q_0}$ denotes the projection onto the convex feasible set $Q_0=\{q:q_1,q_2\geq\sqrt{2\alpha\beta}\}$. Inserting the possible cases $\bar q_i \in\{\alpha v_i,\sqrt{2\alpha\beta}\}$, $i=1,2$, yields
\begin{equation}
    g^{**}_0(v) \in \begin{cases}  
        \frac\alpha2 (v_1^2+v_2^2) + \beta & \text{ if } v_1,v_2 \geq \sqrt{\tfrac{2\beta}\alpha},\\
        \frac\alpha2 v_1^2 + \sqrt{2\alpha\beta}v_2 & \text{ if }  v_1 \geq \sqrt{\tfrac{2\beta}\alpha}\geq v_2,\\
        \frac\alpha2 v_2^2 + \sqrt{2\alpha\beta}v_1 & \text{ if }  v_2 \geq \sqrt{\tfrac{2\beta}\alpha}\geq v_1,\\
        \sqrt{2\alpha\beta}(v_1+v_2) - \beta, &\text{ if } v_1,v_2\leq\sqrt{\tfrac{2\beta}\alpha}.
    \end{cases}
\end{equation}

It remains to decide for a given $v\in\R^2$ which is the maximal of the feasible values.
\begin{enumerate}[(i)]
    \item For $v_1,v_2 \geq \sqrt{\frac{2\beta}\alpha}$, we have the three possible values
        \begin{equation}
            g^{**}(v) \in \begin{cases}
                \frac\alpha2 v_1^2 + \sqrt{2\alpha\beta}v_2,\\
                \frac\alpha2 v_2^2 + \sqrt{2\alpha\beta}v_1,\\
                \frac\alpha2 (v_1^2+v_2^2) + \beta,\\
                \sqrt{2\alpha\beta} (v_1+v_2) - \beta.
            \end{cases}
        \end{equation}
        Since $\sqrt{2\alpha\beta}\leq \alpha v_i$, $i=1,2$, and $\beta>0$, the first two are clearly smaller than the third. For the last case, we consider
        \begin{equation}
            \left(\frac\alpha2 (v_1^2+v_2^2) + \beta\right) - \left(\sqrt{2\alpha\beta} (v_1+v_2) - \beta\right) = \left(\frac\alpha2v_1^2 - \sqrt{2\alpha\beta}v_1\right) + \left(\frac\alpha2v_2^2 - \sqrt{2\alpha\beta}v_2\right) +2\beta.
        \end{equation}
        For these values of $v_1,v_2$, the terms in parentheses are monotonously increasing functions of $v_1$ and $v_2$, respectively; the minimimum is thus attained for $v_1=v_2=\sqrt{\frac{2\beta}{\alpha}}$ at $2\beta >0$. Hence, $g^{**}(v) =  \frac\alpha2 (v_1^2+v_2^2) + \beta$.
    \item For $v_1 \geq \sqrt{\frac{2\beta}\alpha}\geq v_2$, the only two distinct cases are
        \begin{equation}
            g^{**}(v) \in \begin{cases}
                \frac\alpha2 v_1^2 + \sqrt{2\alpha\beta}v_2,\\
                \sqrt{2\alpha\beta} (v_1+v_2) - \beta.
            \end{cases}
        \end{equation}
        Considering the difference of these functions as above, we conclude that $g^{**}(v) =  \frac\alpha2 v_1^2 + \sqrt{2\alpha\beta}v_2$.
    \item We argue similarly for $v_2 \geq \sqrt{\frac{2\beta}\alpha}\geq v_1$ to conclude $g^{**}(v) =  \frac\alpha2 v_2^2 + \sqrt{2\alpha\beta}v_1$.
    \item For $v_1+v_2 \leq \sqrt{\frac{2\beta}\alpha}$, we have to compare the two cases 
        \begin{equation}
            g^{**}(v) \in \begin{cases}
                \frac\alpha2 (v_1+v_2)^2,\\
                \sqrt{2\alpha\beta} (v_1+v_2) - \beta.
            \end{cases}
        \end{equation}
        We have
        \begin{equation}
            \frac\alpha2 (v_1+v_2)^2 - \left(\sqrt{2\alpha\beta} (v_1+v_2) - \beta\right) = 
            \left(\sqrt{\frac{\alpha}{2}}(v_1+v_2) -\sqrt{\beta}\right)^2 \geq 0
        \end{equation}
        and thus  $g^{**}(v) = \frac\alpha2 (v_1+v_2)^2$.
    \item In the remaining case $v_1,v_2\leq \sqrt{\frac{2\beta}\alpha}$ and $v_1+v_2\geq\sqrt{\frac{2\beta}\alpha}$, the only possible value is
        \begin{equation}
            g^{**}(v) = \sqrt{2\alpha\beta} (v_1+v_2) - \beta.
        \end{equation}
\end{enumerate}
Arguing similarly for the three remaining quadrants of $\R^2$, we obtain
\begin{equation}\label{eq:gbiconj}
    g^{**}(v) = \begin{cases}
        \frac\alpha2 (|v_1|^2+|v_2|^2) + \beta & \text{if } v\in D_0,\\
        \frac\alpha2 |v_1|^2 + \sqrt{2\alpha\beta}|v_2| & \text{if } v\in D_1,\\
        \frac\alpha2 |v_2|^2 + \sqrt{2\alpha\beta}|v_1| & \text{if } v\in D_2,\\
        \sqrt{2\alpha\beta} (|v_1|+|v_2|) - \beta & \text{if }v\in D_3,\\
        \frac\alpha2 (|v_1|+|v_2|)^2 & \text{if } v\in D_4,
    \end{cases}\tag{\textsc{b}.1}
\end{equation}
where
\begin{align}
    D_0 &:= \set{v}{|v_1|,|v_2| \geq \sqrt{\tfrac{2\beta}\alpha}},\\
    D_1 &:= \set{v}{|v_1|\geq \sqrt{\tfrac{2\beta}\alpha}\geq |v_2|},\\
    D_2 &:= \set{v}{|v_2|\geq \sqrt{\tfrac{2\beta}\alpha}\geq |v_1|},\\
    D_3 &:= \set{v}{|v_1|,|v_2|\leq \sqrt{\tfrac{2\beta}\alpha},\quad |v_1|+|v_2| \geq \sqrt{\tfrac{2\beta}\alpha}},\\
    D_4 &:= \set{v}{|v_1|+|v_2| \leq \sqrt{\tfrac{2\beta}\alpha}},
\end{align}
see \cref{fig:gbiconj}.
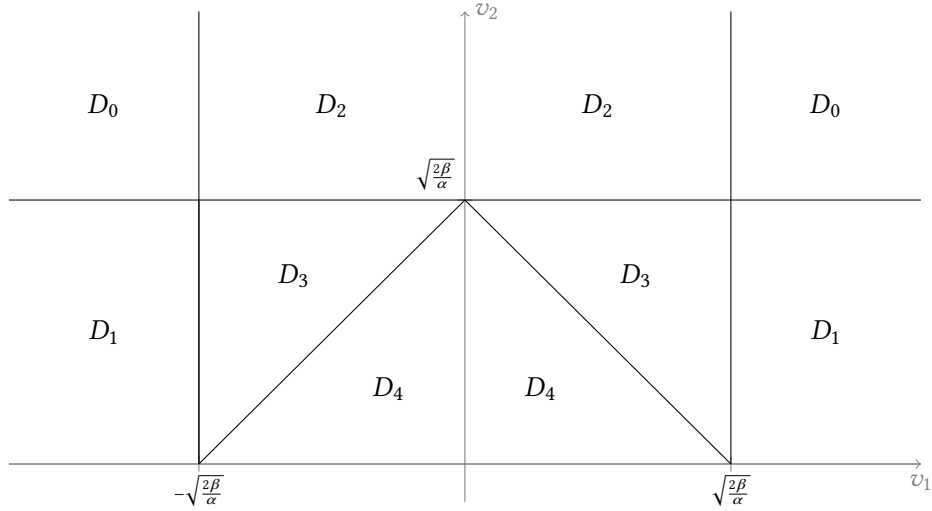
\begin{figure}
    \centering
    \begin{tikzpicture}[]
    \draw[gray,->](0,-2.5) -- (0,4) node[right] {\small $v_2$};
    \draw[gray,->](-6,-2) -- (6,-2) node[below] {\small $v_1$};
    \draw[gray](-6,-2) -- (6,-2);
    \draw[gray](3.5,-1.9) -- (3.5,-2.1);
    \draw[gray](-3.5,-1.9) -- (-3.5,-2.1);
    \draw[gray](-0.1,1.5) -- (0.1,1.5);
    \draw(-3.5,-2) -- (0,1.5);
    \draw(3.5,-2) -- (0,1.5);
    \draw(0,1.5) -- (3.5,1.5);
    \draw(-3.5,1.5) -- (0,1.5);
    \draw(3.5,1.5) -- (6,1.5);
    \draw(-3.5,1.5) -- (-6,1.5);
    \draw(-3.5,-2) -- (-3.5,1.5);
    \draw(3.5,-2) -- (3.5,1.5);
    \draw(-3.5,-2) -- (-3.5,1.5);
    \draw(-3.5,1.5) -- (-3.5,4);
    \draw(3.5,1.5) -- (3.5,4);
    \draw (-4.75,2.75) node {$D_0$};
    \draw (4.75,2.75) node {$D_{0}$};
    \draw (4.75,-0.25) node {$D_{1}$};
    \draw (-4.75,-0.25) node {$D_{1}$};
    \draw (-1.75,2.75) node {$D_{2}$};
    \draw (1.75,2.75) node {$D_{2}$};
    \draw (1,-1) node {$D_{4}$};
    \draw (-1,-1) node {$D_{4}$};
    \draw (2.25,0.5) node {$D_{3}$};
    \draw (-2.25,0.5) node {$D_{3}$};

    \draw (3.5,-2) node[below] {\tiny $\sqrt{\frac{2\beta}{\alpha}}$};
    \draw (-3.5,-2) node[below] {\tiny $-\sqrt{\frac{2\beta}{\alpha}}$};
    \draw (0,1.5) node[above left] {\tiny $\sqrt{\frac{2\beta}{\alpha}}$};

\end{tikzpicture}
    \caption{Subdomains $D_i\subset\R^2$ for the definition of $g^{**}$.}
    \label{fig:gbiconj}
\end{figure}

A short calculation shows that 
\begin{equation}\label{eq:gbiconj_bound}
    g^{**}(v) \geq \frac\alpha2 \left(|v_1|^2 + |v_2|^2\right)\qquad \text{for all } v\in\R^2.\tag{\textsc{b}.2}
\end{equation}
This is obvious for $v\in D_0$ and $v\in D_4$. For $v\in D_1$, we have $\sqrt{2\alpha\beta}\geq \alpha |v_2|$ and hence
\begin{equation}
    g^{**}(v) \geq \frac\alpha2 |v_1|^2 + \alpha|v_2|^2 \geq \frac\alpha2  |v_1|^2 + \frac\alpha2|v_2|^2,
\end{equation}
and similarly for $v\in D_2$. For $v\in D_3$, we consider the difference
\begin{equation}
    \begin{aligned}
        r(v) &:= \left(\sqrt{2\alpha\beta} \left(|v_1|+|v_2|\right) - \beta\right) - \frac\alpha2 \left(|v_1|^2 + |v_2|^2\right)\\
             &= \left(\sqrt{2\alpha\beta}|v_1|-\frac\alpha2|v_1|^2\right) + \left(\sqrt{2\alpha\beta}|v_2|-\frac\alpha2|v_2|^2\right) - \beta.
    \end{aligned}
\end{equation}
On $D_3$, the terms in parentheses are monotonically increasing functions of $|v_1|$ and $|v_2|$ respectively, and thus the minimum is attained at the boundard $|v_1|+|v_2| =  \sqrt{2\beta/\alpha}$, i.e., for $|v_1| = t\sqrt{2\beta/\alpha}$ and $|v_2| = (1-t)
\sqrt{2\beta/\alpha}$ for some $t\in[0,1]$. Inserting this and simplifying yields
\begin{equation}
    r(v) = \beta(2t-2t^2),
\end{equation}
which is a concave quadratic function of $t$ and thus attains its minimum at $t=0$ or $t=1$, yielding $r(v)\geq 0$ as desired.

\section*{Acknowledgments}
The work of CC and KK was supported in part by the Austrian Science Fund (FWF) under grant SFB \textsc{f}32 (SFB ``Mathematical Optimization and Applications in Biomedical Sciences''). The work of KI was partially supported by the Army Research Office under grant \textsc{daad}\,19-02-1-0394.

\printbibliography

\end{document}